\documentclass{amsart}

\usepackage{amsfonts,amsmath,amstext,amsbsy,amsthm,amscd,amssymb}
\usepackage{graphicx}
\usepackage[applemac]{inputenc}

\newcommand{\ligne}[2]{\qbezier(#1)(#1)(#2)}

\theoremstyle{plain}
\newtheorem{theorem}{Theorem} \newtheorem{corollary}{Corollary}
\newtheorem{lemma}{Lemma}
\newtheorem{prop}{Proposition}

\newtheorem*{assumption}{Assumption}
\newtheorem{property}{Property}
\theoremstyle{definition}
\newtheorem{remark}{Remark}
\theoremstyle{remark}
\begin{document}

\title[Convergence of the structure function of an MRW]{Convergence of the structure function of a multifractal random walk in a mixed
  asymptotic setting}
\author[Laurent Duvernet]{Laurent Duvernet \\{\itshape  Universit\'e Paris-Est and \'Ecole Polytechnique}} 
 \address{Universit\'e Paris-Est, Laboratoire d'Analyse et de
  Math\'ematiques Appliqu\'ees, CNRS-UMR 8050, 5, boulevard Descartes,
  77454 Marne-la-Vall\'ee Cedex 2, France\\and  CMAP UMR 7641, \'Ecole
  Polytechnique CNRS, Route de Saclay,  91128 Palaiseau Cedex France}
 \email{duvernet@cmap.polytechnique.fr}
\date{\today}

 \begin{abstract} Some asymptotic
  properties of a Brownian motion in multifractal time, also called
  multifractal random walk, are established. We show the almost sure
  and $L^1$ convergence of its
  structure function.  This is an issue directly
  connected to the scale invariance and multifractal property of the
  sample paths. We place ourselves in a mixed asymptotic setting where
  both the observation length and the sampling frequency may go
  together to infinity at different rates. The results we obtain are similar to the ones
  that were given by Ossiander and Waymire \cite{OW00} and Bacry
  \emph{et al.} \cite{BGHM} in the simpler framework of Mandelbrot cascades.
 \end{abstract}

 \maketitle

\section{Introduction}
Multifractal random processes have become quite popular since the
last two decades, notably in fully developed turbulence (\cite{FP85},
\cite{F95}, \cite{Ga87}, \cite{GMC94})
or
finance (\cite{Ma97}, \cite{BoPo03}, \cite{BMD00}) among other fields. This popularity comes from
the observation of what is often called a multifractal scaling behavior, or
multifractal scale invariance,
 in the data:  Given some observation horizon $t>0$ and some
real-valued data $\bigl( f(x), \, x \in [0,t] \bigr)$, the
structure function of the data simply refers to the
empirical $p$-th moments of the fluctuations $|f(x+l)
-f(x)|$ at a small scale $l>0$. Then the scaling property of the data can be defined as the
power-law behavior of this structure function, which means that the relation
\[ \frac 1 {\lfloor t/l \rfloor}  \sum_{k=1}^{\lfloor t/l \rfloor} \left|f\bigl((k+1)l\bigr) -f\bigl(kl\bigr) \right|^p \approx c(p)
l^{\zeta(p)} \, \text{ as } l \to 0\]
holds for a
variety of exponents $p>0$. Here, $\lfloor a \rfloor$ is the integer part of the positive real number $a$. When the scaling
exponent $\zeta$ is nonlinear, one speaks of  multifractal scaling.

Numerous observations of multifractal scaling have motivated the mathematical
study of functions that satisfy this property. In particular, a large amount of
 work (see notably Jaffard \cite{Jaf00}) has been devoted  to the
 so-called Frisch-Parisi conjecture \cite{FP85} which establishes a
   link between
 the scaling exponent $\zeta$ and
 the regularity of the signal $f$ taken as a function on
 the interval $[0,t]$: according to this conjecture,
  if $D(h)$ is the Hausdorff dimension
 of the level set of the points $x$ where $f$ exhibits a given 
 local H\"older exponent $h$, then $D$ and $\zeta$ are related to one another by a Legendre transform.
 
If we now wish to model such data by a  real-valued random process $X=\bigl(X(t), \, t \geq 0
\bigr)$ with stationary increments, the moments of this process should
have a multifractal scaling. That is:
\begin{property}[Scaling of the moments] \label{pro:1} There exists a
  real-valued nonlinear function
  $\zeta$ defined on a nonempty subset $E_1 \subseteq (0, +\infty)$ such that
\[
 l^{-\zeta(p)} \mathbb E\bigl[ |X(l)-X(0)|^p\bigr] \to c(p) \text{
 as }  l \to 0
\]
for all $p \in E_1$ and some positive numbers $c(p)$.
\end{property}
Moreover, the structure function (the empirical moments) should have the
same scaling property. In this paper, we consider the structure function taken on dyadic
increments: $l=2^{-n}$, $n \ge 0$.  We also
place ourselves in a mixed asymptotic setting where the observation
horizon  may be fixed or may grow as
$t2^{n\chi}$ for some fixed numbers $\chi \geq 0$ and $t>0$; we give
incentives to do so below. Thus,   we wish that $X$ has the
following property:
\begin{property}[Scaling of the structure function] \label{pro:2}
  Assume that Property \ref{pro:1} holds for $\zeta$ defined on
  $E_1$. For $\chi \geq 0$, there exists a nonempty subset $E_2
  \subseteq E_1$, which possibly depends on $\chi$, such that for $t>0$ and $p \in E_2$,
  the renormalized structure function
 \[
    2^{n(\zeta(p)-1-\chi)}  \sum_{k=1}^{\lfloor  t2^{n(1+\chi)}
  \rfloor} |X\big((k+1)2^{-n}\bigr)-X\big(k2^{-n}\bigr)|^p 
 \]
 converges to a positive finite limit as  $n$ goes to $+\infty$.
\end{property}
Finally, the logarithm of this structure function should provide a
consistent estimator of the exponent $\zeta$. Indeed, when dealing
with real data, the  multifractal nature of the data is generally
characterized through a nonlinear behavior of this logarithm. This
gives the following property:
\begin{property}[Estimation of the scaling exponent] \label{pro:3}
  Assume that Property \ref{pro:1} holds for $\zeta$ defined on
  $E_1$. For $\chi \geq 0$, there exists a nonempty subset $E_3
  \subseteq E_1$, which possibly depends on $\chi$, such that for $t>0$ and $p \in E_3$,
 \[
    \frac{ \log_2  \sum_{k=1}^{\lfloor  t2^{n(1+\chi)}
  \rfloor} \bigl|X\big((k+1)2^{-n}\bigr)-X\big(k2^{-n}\bigr)\bigr|^p}{-n} \to
  \zeta(p)-1-\chi 
  \text{ as }   n \to +\infty.
 \]
\end{property}
Remark that if  Property
\ref{pro:2} holds with almost sure convergence and a set $E_2$, then clearly Property  \ref{pro:3} holds with almost sure
convergence and a set $E_3$ such that $E_2 \subseteq E_3$. However, it may be the case that the reverse
inclusion $E_3 \subseteq E_2$ is not true.

This paper is devoted to the study of Properties \ref{pro:2} and
\ref{pro:3}  when $X$
belongs to the class of Multifractal Random Walks (MRW) defined by
Bacry and Muzy in \cite{BM03}. We  give the modes of convergence
and define below the sets $E_1$,
$E_2$ and $E_3$ mentionned in this properties; they will be
open intervals in  $(0, +\infty)$ with $E_2 = E_3$. We also prove that
they are almost maximal in the sense
that if $p$ is larger than the supremum of the interval, then the properties
do not hold.  

By an MRW, we mean a continuous time random process of the form 
\[X(t) = B\bigl(M(t)\bigr), \quad t \geq 0,\]
where $B=\bigl(B(t), \, t \geq 0\bigr)$ is a standard Browian motion,
$M=\bigl(M(t), \, t \geq 0\bigr)$ is a  cascade
process  in the sense of Bacry and Muzy in \cite{BM03}, and $B$ and $M$ are
independent. The process $M$ is positive, nondecreasing, possesses
stationary increments; it is also called Multifractal Random
Measure (MRM) by Bacry and Muzy. Its moment of order $p>0$ satisfies Property
\ref{pro:1} whenever the moment is finite, from which we see  that the process $X$ also satifies Property \ref{pro:1}.  By an argument based on the scaling
property of the Brownian motion $B$, we will see that the convergence of
the structure function of $X$ is
directly connected to the convergence of
the structure function of $M$.

Let us describe the connections between this paper and the work of
other authors. The best known examples of processes that satisfy
Property \ref{pro:1} are Mandelbrot
cascades (\cite{Ma74}, \cite{KP76})  which are constructed by iterated multiplication
of positive i.i.d. random variables on a $b$-adic grid for some fixed
integer $b$. Such
processes also satisfy   Properties \ref{pro:2} and \ref{pro:3} as was
shown by Molchan \cite{Mol96} (for convergence in probability) and
 Ossiander and Waymire \cite{OW00} (for almost sure convergence); however both properties only hold
 when the structure function is taken on $b$-adic increments with the
 same $b$ that is used in the definition of the process. The
 simplicity of the construction of these cascades indeed
 has the drawback that $b$-adic and non $b$-adic increments have
 fundamentally different properties. The MRM of Bacry and Muzy
 is based on one of the continuous analogues of the construction of Mandelbrot
cascades, where the product of i.i.d. random variables is
replaced by the exponential of a L\'evy process, so that the
increments are indeed stationary. To this extent, our
results give a  generalization of the convergence obtained by Ossiander and Waymire. In particular, our choice of considering
dyadic increments for the structure function is somewhat arbitrary and
could for instance be easily replaced by $b$-adic increments for any
integer $b \ge 2$. 

The results of Ossiander and Waymire were proved in a ``fine
resolution'' setting where the discretization step $b^{-n}$ goes to
zero whereas the observation horizon is fixed (i.e. $\chi=0$ with the
notations of Property \ref{pro:2}).  However, it is not obvious that
this asymptotic setting should always be the best for handling a large
number of data. Indeed, an important feature of Mandelbrot cascades
and Multifractal Random Walks is the parameter $T>0$ called
integral scale, which plays the role of a decorrelation time: two
increments of the process are independent as soon as they are taken on
intervals which lie at a distance greater than $T$. The behavior of
the structure function will then be clearly different depending on the
fact that the observation horizon $t2^{n\chi}$ is less than $T$ or much
greater. The latter can notably happen in the case of
turbulence study where many integral scales may be observed.  A recent work by Bacry \emph{et al.}
\cite{BGHM} revisits the convergence of the structure
function of a Mandelbrot cascade in a ``mixed asymptotic'' setting where $\chi$ is positive. Then
 the sets $E_2$ and $E_3$ in properties \ref{pro:2} and \ref{pro:3} nontrivially depend
on the parameter $\chi \in [0, \infty)$. In particular, the authors
  show that the set $E_3$ is nondecrasing with $\chi$, so that by averaging over $t2^{n\chi}$ independent
integral scales with a large $\chi$, one may recover  more exponents
$\zeta(p)$ through the convergence stated in Property \ref{pro:3}.  We extend these results to
the MRW framework: we prove
Property \ref{pro:2} in this mixed asymptotic setting and show that the
regimes for recovering the exponent $\zeta$ in Property \ref{pro:3} are the same for
MRW's and Mandelbrot cascades.

Whereas Property \ref{pro:1}  was already shown by Bacry and
Muzy in \cite{BM03} (actually the relation in Property \ref{pro:1} is an exact
equality for all $l \leq T$), Property \ref{pro:2} has not yet been studied in the case of MRW's, with the
exception of a recent work by Lude\~na \cite{Lud08} which investigates the case of integer values of the exponent $p$ in a slightly different
framework than ours since $M$ is integrated with
respect to a fractional Brownian motion with Hurst parameter $H \in(1/2,3/4)$.

The paper is organized as follows. In Section \ref{sec:results}, we recall the construction of an MRW,
state and discuss our results.
Sections \ref{sec:dem1} and \ref{sec:dem2} are respectively devoted to
the proofs of Theorems \ref{thm:CvS} and \ref{thm:CvMixte}  that 
state  Property \ref{pro:2}  in the fine
resolution ($\chi=0$) and mixed asymptotic ($\chi>0$) settings. The limit is a nondegenerate
random variable in the first case and a deterministic value in the
second case. Section \ref{sec:dem3} consists in the
proof of Theorem \ref{thm:CvZeta} that states Property \ref{pro:3} and
 the maximality of the sets $E_2$ and $E_3$.
 Finally, some technical proofs are presented in the appendices.

\section{Definitions and results} \label{sec:results}

\subsection{Construction  of $M$ and $X$}
Let us recall the construction of the MRM  as it is described by
Bacry and Muzy in \cite{BM03}. We first fix a number $T>0$ that is
the integral scale of the process and an infinitely divisible distribution
$\pi(dx)$ on $\mathbb R$. 
Let $\psi$ be the Laplace exponent of $\pi$:
 \begin{equation*}
  e^{\psi(q)} = \int_\mathbb R e^{qx} \pi(dx)
 \end{equation*}
for $q \geq 0$ (possibly $\psi(q) = \infty$). The well known
L\'evy-Khinchine formula states that whenever $\psi(q) < +\infty$,  $\psi(q)$ is of the following
form: 
\[\psi(q) = aq + \frac 1 2 \sigma^2 q^2 + \int
\bigl(e^{qx}-1-qx1_{\{|x|<1\}}\bigr)m(dx),\]
where $a$ and $\sigma$ are real numbers and $m(dx)$ is a Borel measure
on $\mathbb R$ that satisfies
\[\int \min(x^2,1)m(dx) < +\infty.\] 
For $q \ge 1$, we define the following condition on $\pi(dx)$:
\begin{assumption}[$\mathbf{A}_q$] The positive number $q$ and the infinitely divisible distribution
$\pi(dx)$ are such that
\[\int  \min(|x|,1)m(dx)  <  +  \infty,  \]\[ \psi(1)=0, \]
and
\[\psi\bigl(q(1+\epsilon)\bigr) < q(1+\epsilon)-1 \text{ for some } \epsilon>0.\] 
\end{assumption}
Note that since $\psi$ is convex and satisfies $\psi(0)=\psi(1) =0$ under
 $\mathbf A_q$, it is an increasing function on $[1,+\infty)$, so that for
$1<q_1<q_2$, $\mathbf{A}_{q_2} \Rightarrow \mathbf{A}_{q_1}
\Rightarrow \mathbf{A}_1$.

 Let $\mu$ be the measure  on the open half-plane $\mathbb R
 \times (0,\infty)$ given by
  $\mu(dt,dl) = l^{-2}dt \otimes dl$.
 One can define (see Rajput and Rosinski \cite{RR89}) an infinitely divisible,
 independently scattered random measure  $P$ on $\mathbb R
 \times (0,\infty)$ that has an intensity $\mu$ and a Laplace exponent $\psi$, that is:
 \begin{itemize}
  \item  for every Borel set $ B$  in $\mathbb R
 \times (0,\infty)$, $ P( B)$
  is an infinitely divisible random variable such that:
  \begin{equation}
   \mathbb E \left[ e^{q P( B)} \right] = e^{\mu(B)
   \psi(q)}
  \end{equation}
  for every $q \geq 0$ such that $\psi(q) < \infty$,
  \item for every sequence $\{B_k\}_{k \in \mathbb N}$ of disjoint Borel sets in $\mathbb R
 \times (0,\infty)$, the variables $ P(B_k)$ are
  independent and
  \begin{equation*}
   P (\cup_{k \in \mathbb N}B_k) = \sum_{k \in \mathbb N} 
  P(B_k) \text{ almost surely}.
  \end{equation*}
 \end{itemize}
 The  process $\omega = \bigl(\omega_l(t), \, (t,l) \in \mathbb R \times (0,\infty)\bigr)$ is defined by
 $
 \omega_l(t) =  P(\mathcal A_l(t))
 $
 for $(t,l) \in \mathbb R \times (0,\infty)$, where
 \begin{equation*}
 \mathcal A_l(t) = \left\{ (t',l') \in \mathbb R \times (0,\infty) , \, l \leq l' \text{ and } |t-t'| \leq \frac 1 2 \min (l', T)
 \right\}.
 \end{equation*}

\setlength{\unitlength}{.7cm}

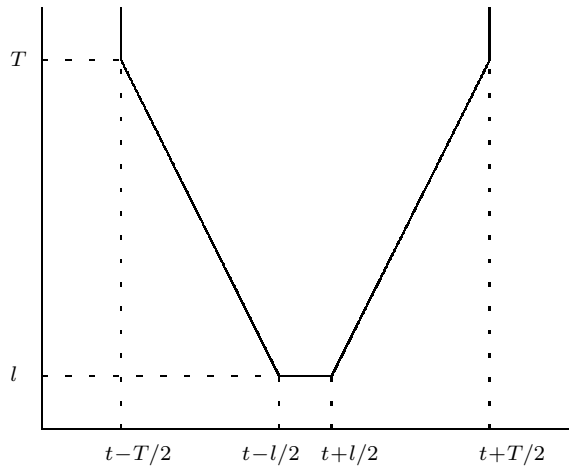
\begin{figure}[htbp]
  \centering

   \begin{picture}(11,9)(-1,-1)
    \put(0,0){\line(1,0){10}}
    \put(0,0){\line(0,1){8}}
    \put(-.6,.9){\footnotesize $l$}
    \multiput(0,1)(.45,0){10}{\line(1,0){.1}}
    \put(-.6,6.9){\footnotesize $T$}
    \multiput(0,7)(.45,0){4}{\line(1,0){.1}} 
    \put(3.8,-.6){\footnotesize $t\!-\!l/2$}
    \put(5.3,-.6){\footnotesize $t\!+\!l/2$}
    \put(1.2,-.6){\footnotesize $t\!-\!T/2$}
    \put(8.3,-.6){\footnotesize $t\!+\!T/2$}
    \multiput(1.5,0)(0,.45){16}{\line(0,1){.07}}
    \multiput(4.5,0)(0,.45){3}{\line(0,1){.07}}
    \multiput(5.5,0)(0,.45){3}{\line(0,1){.07}}
    \multiput(8.5,0)(0,.45){16}{\line(0,1){.07}}
    \ligne{4.5,1}{5.5,1}
    \ligne{4.5,1}{1.5,7}
    \ligne{1.5,7}{1.5,8}
    \ligne{5.5,1}{8.5,7}
    \ligne{8.5,7}{8.5,8}
  \end{picture}
 
 \caption{The cone $\mathcal A_l(t)$}

\end{figure}

 The following essential result is borrowed from \cite{BM03}:
 \begin{prop}[Existence of the MRM \cite{BM03}] \label{thm:BM03}
   If Assumption $\mathbf{A}_1$ holds,  then $M_l(t) = \int_0^t e^{\omega_l(u)}
 du$ converges as $l\rightarrow 0$ almost surely and in $L^1$ to a
   random variable $M(t)$.
 \end{prop}
It is then clear from the $L^1$ convergence and the
condition $\psi(1)=0$ that $\mathbb E[M(t)]=t$.
\begin{remark}
The condition on the measure $m(dx)$ is not mentionned in the original
paper of Bacry and Muzy \cite{BM03}. However, Barral and Mandelbrot \cite{BaMa03} noticed that it
is needed for Proposition \ref{thm:BM03} to hold. The present
work does not make any further explicit references to this condition besides assuming that
the probability distribution $\pi(dx)$ is such that Proposition \ref{thm:BM03} holds.
\end{remark}
 Under Assumption $\mathbf{A}_1$, we can define an MRW by setting $X(t) = B\bigl(M(t)\bigr)$ where $B$ is a standard
 Brownian motion independent of $M$. Then $M$ and $X$  are 
 random processes with continuous paths and stationary increments. Let
 us stress that by construction, two increments  $M(b)-M(a)$ and
 $M(d)-M(c)$ with $a < b <c <d$ are independent as soon as $|b-c| \geq
 T$. Obviously, the same also holds for the increments of $X$.

\subsection{The moments of order $q \ge 1$ of $X$ and $M$}

 Let us give the criterion for the existence of the moments of $M$
 and $X$, also borrowed from
 \cite{BM03}:

 \begin{prop}[Moments of the MRM \cite{BM03}]
   \label{thm:BM03moments}   If Assumption $\mathbf{A}_q$ holds for
   some $q>1$,  then for $\epsilon>0$ such that
   $\psi(q(1+\epsilon))<q(1+\epsilon)-1$, $M_l(t)$ converges  in $L^{q(1+\epsilon)}$ as $l \to 0$.
  In this case, the $ r$-th moment of $M(t)$ for $t \in [0,T]$ and
 $r \in
 [0,q(1+\epsilon)]$ is given by
 \[ 
 \mathbb E \bigl[ M(t)^{r} \bigr]= \gamma(r) T^{\psi(r)} t^{r-\psi(r)}
 \]
 where $\gamma(r)$ is positive and does not depend on $t$
 ($\gamma(0)=\gamma(1)=1$).

 Conversely,
 if $\mathbf{A}_1$ holds and 
 $\psi(q) > q-1$  for some  $q>1$ (so that $\mathbf{A}_q$ does not
 hold), then $\mathbb E
 \bigl[M(t)^q\bigr] = +\infty.$
 \end{prop}

\begin{remark} \label{rem:pro1}
 Under Assumption  $\mathbf{A}_1$, let us define 
\begin{equation}
 q^* = \sup \{q \ge 1 , \,\mathbf{A}_q \text{ holds}\}.
\end{equation}
 From the scaling property of the Brownian motion $B$, the process $X$ then satisfies Property \ref{pro:1} with a
 scaling exponent $\zeta^X:p \mapsto p/2-\psi(p/2)$ defined on the set $E_1=(0,2q^*).$  This function is non linear as soon as the
 infinitely divisible distribution $\pi(dx)$ is non degenerate. Note that depending on  $\pi(dx)$, it
may be the case that $q^* = +\infty$ (for instance if $\pi(dx)$ is a
Poisson distribution, but not if $\pi(dx)$ is a Gaussian distribution, see
\cite{BM03}). 
\end{remark}

\subsection{The structure functions $S_n$ and $\Sigma_n$}
For $t >0$ and $\chi \geq 0$, we define $t_n = 2^{n\chi}t$ that
 is our observation horizon.  
 Our object of interest is the structure function of $X$ which we
 define on a dyadic sampling $k2^{-n}, \, n \in \mathbb N$, that is:
\begin{equation*} \label{eqn:defSn}
S_n(2q, t, \chi) = \sum_{k=0}^{\lfloor t2^{n(1+\chi)} \rfloor-1}  \bigl|X\bigl((k+1)2^{-n}\bigr) -
X\bigl(k2^{-n}\bigr)\bigr|^{2q} , \, t>0, \, q \in [0,+\infty).
\end{equation*}
We define $b_{n,k}$ for $0 \leq k \leq \lfloor t2^{n(1+\chi)} \rfloor-1$ and $n
\in \mathbb N$ as 
\begin{equation} \label{eqn:defb}b_{n,k} = M\bigl((k+1)2^{-n}\bigr) -
  M\bigl(k2^{-n}\bigr).\end{equation} Then, using the scaling property
of the Brownian motion, we see that $S_n(2q, t, \chi)$ has the same law as
\begin{equation*}
\sum_{k=0}^{\lfloor t2^{n(1+\chi)} \rfloor-1}  |\xi_k|^{2q} b_{n,k}^{q} 
\end{equation*}
where the $\xi_k$'s are i.i.d. standard normal random variables
independent of $M$. In particular, if we define $\Sigma_n(q, t , \chi)$ as
 \begin{equation*} \label{eqn:defSigman}
  \Sigma_n(q, t , \chi) = \sum_{k=0}^{\lfloor t2^{n(1+\chi)} \rfloor-1}  b_{n,k}^{q},
 \end{equation*}
then under  Assumption $\mathbf A_q$ and as soon as $2^{-n} \le T$, we have from Proposition
\ref{thm:BM03moments}:
 \begin{equation} \label{eqn:EMRM}
 \mathbb E\bigl[ \Sigma_n(q,t,\chi) \bigr]= \gamma(q)  T^{\psi(q)}
 \lfloor t 2^{n \chi}\rfloor 2^{n(\psi(q)-q)}
 \end{equation}
 and 
 \begin{equation} \label{eqn:EMRW}
 \mathbb E\bigl[ S_n(2q,t,\chi) \bigr] = \mathbb E\bigl[|\xi_0|^{2q} \bigr] \mathbb E\bigl[ \Sigma_n(q,t,\chi) \bigr].
 \end{equation}

We will study the behavior of $S_n(2q, t ,\chi)$ and $\Sigma_n(q,t,\chi)$ in
different asymptotics. In the ''fine resolution`` setting, $\chi=0$ 
so that the observation horizon is fixed, while the case  $\chi >0$ defines what we call
the ''mixed asymptotic`` setting.

\subsection{Asymptotic values and regimes}

 For $q>0$ such that $\psi(q) < +\infty$, we introduce a new
 sequence of processes $M_{2^{-n}}^{(q)}(t)$ that will be shown to be an asymptotic
 value of  $S_n$ and $\Sigma_n$. Its definition is similar to the above definition of $M_{l}(t)$.
 We write \[ P^{(q)}(B) = q  P(B) - \log \mathbb E [
  e^{q  P(B)} ] \] for every Borel set $B$ of $\mathbb R
 \times (0,\infty)$. The function $\psi^{(q)}:r \mapsto \psi(qr) - r
 \psi(q)$,  defined for nonnegative $r$'s such that
  $\psi(qr)<+\infty$,  is then the Laplace exponent of
  $P ^{(q)}$.
  In particular, if we set \[\omega_{2^{-n}}^{(q)}(t) =  P^{(q)} (\mathcal
  A_{2^{-n}}(t)) \quad \text{for } t>0, \, n\in \mathbb N,\] then the
 process $\omega^{(q)}$ has the following property:
  \begin{equation*} 
   \mathbb E \bigl[e ^{r \omega_{2^{-n}}^{(q)}(t)} \bigr]=
  e^{\mu(\mathcal A_{2^{-n}}) \psi^{(q)}(r)}= \frac{\mathbb E \bigl[e ^{qr
  \omega_{2^{-n}}(t)}\bigr]}
  {\left( \mathbb E \bigl[e^{q
  \omega_{2^{-n}}(t)}\bigr]\right)^r}\end{equation*}
  for $r$ and $q \ge 0$ such that $\psi(qr)<+\infty.$
   We now define $M_{2^{-n}}^{(q)} (t) = \int_0^t
  e^{\omega_{2^{-n}}^{(q)}(u)} du;$  in particular  we have  $\mathbb
  E[M_{2^{-n}}^{(q)} (t)] =t$.  We finally introduce a condition
 on $q $ and $\chi $ that will define the asymptotic regimes of $S_n$
  and $\Sigma_n$:
 \begin{assumption}[$\mathbf B^{(q)}(\chi)$] The infinitely divisible distribution
$\pi(dx)$ and the real numbers $q >0$ and $\chi \geq 0$ are such that
   \[\psi\bigl(q(1+\epsilon)\bigr)< +\infty  \text{ and } \psi^{(q)}(1+\epsilon) < \epsilon(1+\chi) \text{ for some } \epsilon>0.\]
 \end{assumption}
 It is straightforward  to show from the convexity of the Laplace exponent
 $\psi$ 
 that for $\epsilon >0$, $\psi^{(q)}(1+\epsilon)$ increases with $q$. Thus for $\chi \geq 0$ and  $0<q_1<q_2$, $\mathbf B^{(q_2)}(\chi)
 \Rightarrow \mathbf B^{(q_1)}(\chi) $. Conversely, if $0 < \chi_1 < \chi_2$ and $q>0$, we clearly have $\mathbf B^{(q)}(0)
 \Rightarrow \mathbf B^{(q)}(\chi_1) \Rightarrow \mathbf
 B^{(q)}(\chi_2)$.  

 Note that under Assumptions $\mathbf A_1$  and $\mathbf B^{(q)}(0)$, Proposition \ref{thm:BM03} gives the existence of the
 limit $M^{(q)}(t) = \lim M_{2^{-n}}^{(q)}(t)$ for $q  >0$ and $t>0$, where the convergence is
 almost sure and in $L^1$. However, if only  Assumptions $\mathbf A_1$
 and $\mathbf B^{(q)}(\chi)$ hold for $\chi>0$, then $ M_{2^{-n}}^{(q)}(t)$ does not necessarily have a nondegenerate limit.

\subsection{Statement of the main results}
\begin{theorem}[Convergence of $S_n$ and $\Sigma_n$ in the fine
 resolution setting] \label{thm:CvS}
Suppose that either $q \in (0,1]$ and Assumption $\mathbf A_1$ holds, or $q>1$ and Assumptions $\mathbf A_q$ and $\mathbf B^{(q)}(0)$ hold, then for  $t >0$ 
\[   \frac{S_n(2q,t,0)}{\mathbb E \bigl[S_n(2q,t,0)\bigr]}
\to \frac {M^{(q)} (t)}{t} \quad \text{as} \quad n\to + \infty
  \]
almost surely and in $L^1$.
Moreover, the same result also holds if one replaces $S_n(2q,t,0)$ with $\Sigma_n(q,t,0)$.
\end{theorem}

\begin{theorem}[Convergence of $S_n$ and $\Sigma_n$ in the mixed asymptotic setting] \label{thm:CvMixte}
For $\chi>0$, suppose that either $q \in (0,1]$ and  Assumption $\mathbf A_1$ holds, or $q>1$ and Assumptions $\mathbf A_q$ and $\mathbf B^{(q)}(\chi)$ hold, then for  $t >0$ 
\[   \frac{S_n(2q, t, \chi)}{\mathbb E \bigl[S_n(2q, t, \chi)\bigr]}
\to 1\quad \text{as} \quad n\to + \infty
  \]
almost surely and in $L^1$.
Moreover, the same result also holds if one replaces $S_n(2q,t,\chi)$ with $\Sigma_n(q,t,\chi)$.
\end{theorem}

\begin{remark}
 For $q>1$,
 Proposition \ref{thm:BM03}  shows that if Assumptions $\mathbf A_1$
 (or $\mathbf A_q$)  and $\mathbf B^{(q)}(0)$ hold, then  $M^{(q)}
 (t)$  is well defined and $\mathbb E [M^{(q)} (t)] = t$. In
 particular, the strong  law of large numbers proves that for $\chi>0$
\[\frac{M^{(q)} (2^{n\chi}t)} {\mathbb E[M^{(q)} (2^{n\chi}t)]} \to 1 \quad \text{as} \quad n \to +\infty\]
almost surely and in $L^1$. However, if only Assumptions
$\mathbf A_q$ and $\mathbf B^{(q)}(\chi)$  hold, then $M^{(q)} (t)$ is
not necessarily well defined,  so that Theorem \ref{thm:CvMixte} is
not an  immediate consequence of Theorem \ref{thm:CvS}.

The case $q \in (0,1]$ is simpler. Indeed, one may check that
 Assumption $\mathbf
 A_1$ is the same as  Assumption $\mathbf B^{(1)}(0)$, which implies
 Assumption $\mathbf B^{(q)}(0)$, so that $M^{(q)}
 (t)$ is always well defined in this case.
\end{remark}

For some given $\chi \ge 0$ and infinitely divisible distribution
$\pi(dx)$, we define $q_\chi$ as:
\begin{equation} \label{eqn:qchi} q_\chi = \sup \{ q >0 , \, \mathbf
  B^{(q)}(\chi) \text{ holds}\}.\end{equation}
Under Assumption $\mathbf A_1$, it is clear that $q_\chi \ge 1$. Depending on the distribution $\pi(dx)$, it may be the case
that $q_\chi = +\infty$.

\begin{theorem}[Estimation of the scaling exponent]
  \label{thm:CvZeta}
If $q>0$  and the infinitely divisible distribution $\pi(dx)$ are such
that Assumption $\mathbf A_q$ holds, then
for  $t >0$ and $\chi \ge 0$ the following relations hold almost surely:
 if $q<q_\chi$
\[ \frac{\log_2(S_n(2q,t,\chi))}{-n}  \to
 q-\psi(q)-1 - \chi \quad \text{as} \quad n\to + \infty,\]
and if $q_\chi < +\infty$ and $q \ge q_\chi$
\[ \frac{\log_2(S_n(2q,t,\chi))}{-n}  \to
q(1-\psi'(q_\chi) ) \quad \text{as} \quad n\to + \infty.\]
Moreover, the same results also hold if one replaces $S_n(2q,t,\chi)$ with $\Sigma_n(q,t,\chi)$.
\end{theorem}

The reader will find the proofs of these theorems in the remaining
sections of the paper. Recall now that $q^*$ has been defined as the
supremum of the $q \geq 1$ such that $\mathbf A_q$ holds. The theorems
above allow us to state Properties \ref{pro:1}, \ref{pro:2}, and
  \ref{pro:3} for MRM's and MRW's. We define $\zeta^X:p \mapsto
  p/2-\psi(p/2)$, $E_1^X = (0,2q^*)$, and $E_2^X = E_3^X= (0, 2 \min\{q^*,
  q_\chi\})$, and also $\zeta^M:q \mapsto q - \psi(q)$, $E_1^M =
  (0,q^*)$, and $E_2^M = E_3^M= (0,  \min\{q^*,
  q_\chi\})$.

\begin{corollary} \label{cor:basic} If the infinitely divisible distribution $\pi(dx)$ is such
that Assumption $\mathbf A_1$ holds, 
  then $X$  and $M$ satisfy Properties \ref{pro:1}, \ref{pro:2}, and
  \ref{pro:3}, the convergence being almost sure and in $L^1$ for
  Property \ref{pro:2} and almost sure for Property
  \ref{pro:3}. The function $\zeta$ in this properties  is
  $\zeta^X$ for $X$ and $\zeta^M$ for $M$, and the sets $E_1$, $E_2$
  and $E_3$ are respectively $E_1^X$, $E_2^X$
  and $E_3^X$, and $E_1^M$, $E_2^M$
  and $E_3^M$. Moreover, these sets are almost maximal, in the sense that the
  properties do not hold  for any strictly larger open sets in
  $(0, +\infty)$.
\end{corollary}

\begin{proof}[Proof of Corollary \ref{cor:basic}] 
As we already mentioned (see Remark \ref{rem:pro1}), Property 1 and
the almost maximality of the set $E_1$ has been proved by Bacry and Muzy in \cite{BM03}. From the
definition of $q_\chi$, it is straightforward to check that if
$q_\chi<q$, then \[q-\psi(q)-1 -
\chi \ne q(1-\psi'(q_\chi) ).\] Thus,
Theorem \ref{thm:CvZeta}  clearly implies the statement of 
Corollary \ref{cor:basic} concerning Property \ref{pro:3}. Moreover,
Theorems \ref{thm:CvS} and \ref{thm:CvMixte} state that Property
\ref{pro:2} holds for the set $E_2$, while
Theorem \ref{thm:CvZeta} also proves that Property \ref{pro:2} does not
hold for an open set $E$ such that $E_2 \subset E \subseteq E_1.$ \end{proof}

\begin{remark}
The same results in the framework of
Mandelbrot cascades were obtained by Ossiander and Waymire \cite{OW00}
(in the fine resolution setting) and Bacry \emph{et al.} \cite{BGHM} (in the
mixed asymptotic setting). This could be interpreted in the following way: eventhough MRW's and
MRM's are quite more elaborate objects than Mandelbrot cascades,
they share some essential properties. 
\end{remark}
\begin{remark}
Theorem \ref{thm:CvZeta} shows that from $\log_2(S_n(2q,t,\chi)) / n$,
we can easily obtain a
consistent estimator of $\psi(q), \, q \in (0, \min\{q^*,
q_\chi\})$. Note that in the case $\chi=0$ and $q>1$, one may show  with simple arguments based on the convexity of $\psi$ that if  
$\mathbf A_{1}$ and $\mathbf B^{(q)}(0)$ both hold, then  Assumption
$\mathbf A_{q}$ also holds. Thus the former condition is 
sufficient for the convergence of $S_n$ in Theorem \ref{thm:CvS}. Under
$\mathbf A_1$, we have in particular that if $q_0 < +\infty$, then
$q_0<q^*$, so that the set $E_2$ in Corollary \ref{cor:basic}  increases with $\chi$: one is able to recover more and more values $\psi(q)$ when $\chi$ grows.

In the case $\chi = \infty$ (that is, the resolution scale $2^{-n}$ is
fixed while the observation horizon $t$ goes to infinity), one will then
be able to estimate $\psi(q)$ for all $q \in (0, q^*)$. Indeed,
if we define for $q>0$, $t>0$, and $n \in \mathbb N$: 
\[S_n(2q,t,+\infty) = \sum_{k=0}^{\lfloor t2^n \rfloor-1} \bigl|X\bigl((k+1)2^{-n}\bigr)-X\bigl(k2^{-n}\bigr)\bigr|^{2q},\] then 
two increments $|X\bigl((k+1)2^{-n}\bigr)-X\bigl(k2^{-n}\bigr)\bigr|$ and $|X\bigl((k'+1)2^{-n}\bigr)-X\bigl(k'2^{-n}\bigr)\bigr|$ are independent as soon as
$|k-k'-1|2^{-n}>T$. Thus, we may apply the strong law of large numbers, since for $0<q<q^*$, Proposition \ref{thm:BM03moments} implies that $S_n(2q,t,\infty)$ has a finite expectation. This gives: almost surely, \[   \frac{ 1}{\lfloor t2^n\rfloor} S_n(2q, t, \infty)
\to \mathbb E \bigl[|X(2^{-n}) - X(0)|^{2q}\bigr]\quad \text{as} \quad t\to + \infty.
  \] From the scaling property of the Brownian motion and Proposition
  \ref{thm:BM03moments} (assuming that $n$ is such that $2^{-n} \le T$), this limit is \[\mathbb E \bigl[|X(2^{-n}) - X(0)|^{2q}] =  a(2q) \gamma(q) T^{\psi(q)} 2^{n(\psi(q)-q)},\] where $a(2q)$ is the absolute moment of order $2q$ of a standard normal random variable. Therefore, if we choose two different values $n_1$ and $n_2$ in $\mathbb N$, then almost surely
\[\frac{\log_2 \bigl(S_{n_1}(2q,t,+\infty)\bigr) - \log_2 \bigl(S_{n_2}(2q,t,+\infty)\bigr)}{n_2-n_1} \to q-\psi(q) - 1 \quad \text{as} \quad t \to +\infty.\]
\end{remark}
\begin{remark}
It would be
interesting to obtain convergence rates of this estimator in the case
$\chi \in (0,+\infty)$. However,
empirical evidence  (see Bacry \emph{et al.} \cite{BMK08} and Wendt \emph{et al.} \cite{WAJ07}) suggests that more elaborate estimators attain faster
rates and should be used in practice. 
\end{remark}
\begin{remark}
A signal is of multifractal regularity if it exhibits several local
H\"older exponents $h$ on sets of positive Hausdorff dimension $D(h)$.  The
so-called ``multifractal formalism'' claims that $D(h)$ is the Legendre
transform of the exponent $q \mapsto \zeta^M(q)+1=\psi(q)-q+1$ of the structure function $\Sigma_n(q,t,0)$.  In the framework of Mandelbrot cascades, Bacry \emph{et al.} \cite{BGHM} define $D(h)$ as a box-counting
dimension. From an analogue of Theorem \ref{thm:CvZeta} and the fact that the multifractal formalism holds for Mandelbrot cascades in the fine resolution setting $\chi=0$ (see Molchan \cite{Mol96}), they show that in the setting of the mixed asymptotic, $D(h)$
is  the Legendre transform of $ \psi(q)-q +1+ \chi$. Presumably, the same also holds for MRW's. Notice however
that as of today, the multifractal formalism has not
been fully proved in the framework of MRW's even in the setting of the
fine resolution asymptotic 
 (see Barral and Mandelbrot \cite{BaMa03} for a state of the art).
\end{remark}

\vspace{\baselineskip}
  We write $u_n \lesssim v_n$
if there exists some real (non-random) number $c>0$ such that
\[ \forall \, n , \; u_n \leq c v_n\]
and $u_n \asymp v_n$
if there exist some real (non-random) numbers $c_1, \, c_2>0$ such that
\[ \forall \, n , \; c_1 v_n \leq u_n \leq c_2 v_n.\]

The symbol $\stackrel{d}{=}$ denotes equality in distribution.

\section{Proof of Theorem \ref{thm:CvS}} \label{sec:dem1}

\subsection{Outline of the proof}

The proof
is separated in two steps. We first prove
Proposition \ref{lem:MRW} which states  that  $S_n(2q,t,\chi)$
and $\Sigma_n(q,t,\chi)$ are asymptotically equal.  Next, we prove Proposition
\ref{lem:Approximation} which gives a precise upper bound for the term
\[\left| \frac{\Sigma(q,t,0)}{\mathbb E [\Sigma_n(q,t,0)]} -
\frac{M_{2^{-n}}^{(q)} (t)}{\mathbb E [M_{2^{-n}}^{(q)} (t)]}\right|
. \]
Under Assumptions $\mathbf A_1$ and $\mathbf B^{(q)}(0)$, we finally apply
Proposition \ref{thm:BM03} to see that $M_{2^{-n}}^{(q)}(t)
\rightarrow M^{(q)}(t)$ almost surely and in $L^1$. This shows Theorem \ref{thm:CvS}.

Note that the statements of these propositions remain valid under
broader assumptions  than those  of Theorem \ref{thm:CvS} (for
instance, they do not require Assumption $\mathbf B^{(q)}(0)$); this
enables us to  use  these two propositions during the proof of  Theorem \ref{thm:CvMixte}. 

\begin{prop} \label{lem:MRW} If $q>0$, $\chi \ge0$, and the infinitely divisible distribution $\pi(dx)$ are such
that Assumptions $\mathbf A_q$ and $\mathbf B^{(q)}(\chi)$ hold, then
for  $t >0$
 \[
  \left| \frac{S_n(2q,t,\chi)}{\mathbb E [S_n(2q,t,\chi)]} -
  \frac{\Sigma_n(q,t,\chi)}{\mathbb E [\Sigma_n(q,t,\chi)]} \right| \to  0\quad \text{as} \quad n\to + \infty
  \]
almost surely and in $L^1$.
\end{prop}

\begin{prop} \label{lem:Approximation}
Let $q>0$ and the infinitely divisible distribution $\pi(dx)$ be such
that Assumption $\mathbf A_q$ holds.  Then there exist
some processes $A_n$ and $B_n$
such that for $t > 0$
\[   \frac{\Sigma_n(q,t,0)}{\mathbb E[ \Sigma_n(q,t,0)]} -
\frac{M_{2^{-n}}^{(q)} (t)}{\mathbb E [M_{2^{-n}}^{(q)} (t)]} \asymp A_n(t) + B_n(t)
,\]
where  the processes $A_n$ and $B_n$ satisfy the following properties:
the sequences 
$\bigl( A_n(k2^{-n}), \,  k \in \mathbb N \bigr)$ and $\bigl( B_n(k2^{-n}), \,  k \in \mathbb N \bigr)$
have stationary increments, these increments are independent as soon
as they are taken on intervals that lie at a distance greater than $T$, \[\mathbb E \bigl[|A_n(t)|\bigr] \lesssim 2^{- n
  \alpha}   \] for some $\alpha>0$,  and 
\[ \mathbb E [B_n(t)] = 0 \quad \text{and} \quad \mathbb E [\,|B_n(t)|^{1+\epsilon}] \lesssim \, 2^{n(\psi^{(q)}
  (1+\epsilon)- \epsilon)} \]
for $\epsilon \in (0,1)$ such that $\mathbb
E[M(t) ^{q(1+\epsilon)}] < +\infty$.
\end{prop}

\subsection{Proof of Proposition \ref{lem:MRW}}
Recall that from
 the scaling property of the Brownian motion,
 \begin{equation*} 
  S_n(2q,t,\chi) \stackrel{d}{=} \sum_{k=0}^{\lfloor 2^{n\chi}t \rfloor-1}  |\xi_k|^{2q} b_{n,k}^q
 \end{equation*}
 where the $\xi_k$'s are i.i.d. standard normal random variables
 independent of $M$. 
From Assumption $\mathbf B^{(q)}(\chi)$, we may choose $\epsilon>0$  such that 
\begin{equation} \label{eqn:epsilonzero}
- \psi\bigl(q(1+\epsilon)\bigr) +  (1+\epsilon) \psi(q)+ \epsilon(1+\chi) > 0.
\end{equation}
We write $a(2q)$ for the absolute moment of
 order $2q$ of the $\xi_k$'s, so that 
\[\mathbb E [S_n(2q,t,\chi)] =  a(2q) \mathbb E [\Sigma_n(q,t,\chi)].\]
We now study the moment of order
 $1+\epsilon$ of \[ \left| \frac{S_n(2q,t,\chi) - a(2q) {\Sigma_n(q,t,\chi)}}{\mathbb E
    [S_n(2q,t,\chi)]}\right| . \]
Factorizing by the increments of $M$, we have:
\begin{equation*}
 \mathbb E \Biggl[\left|  S_n(2q,t,\chi)  -  a(2q) {\Sigma_n(q,t,\chi)}
 \right|^{1+\epsilon} \Biggr]
 =   \mathbb E \Biggl[ \biggl| \sum_{k=0}^{\lfloor  2^{n\chi}t  \rfloor-1}
  \bigl(|\xi_k|^{2q}  -  a(2q)\bigr) b_{n,k}^q  \biggr|^{1+\epsilon}\Biggr]
   .
\end{equation*}
We will use several times  the following inequality:  
let $Y_1,...,Y_n$ be a sequence of martingale increments and fix $\epsilon \in [0,1]$. Then
\begin{equation} \label{eqn:EvB}
\mathbb E \Biggl[ \Bigl|\sum_{k=1}^n Y_k\Bigr|^{1+\epsilon} \Biggr]
\lesssim \sum_{k=1}^n \mathbb E\Biggl[  \left| Y_k \right|^{1+\epsilon} \Biggr]
\end{equation}
(a proof can be found in
\cite{EvB}). If we take \[Y_k = \bigl(|\xi_k|^{2q}-a(2q)\bigr)
b_{n,k}^q ,\] then conditionally on the  sigma-field generated by the
$b_{n,k}$'s, $k = 0, \dots , \lfloor  2^{n\chi}t  \rfloor-1$, it is
clear that the $Y_k$'s are i.i.d. and centered.
 Inequality
\eqref{eqn:EvB} therefore applies:
\begin{equation*}
 \mathbb E \bigl[ \left|  S_n(2q,t,\chi) - a(2q) {\Sigma_n(q,t,\chi)}
 \right|^{1+\epsilon} \bigr]\lesssim \mathbb E [\Sigma_n(q(1+\epsilon),t,\chi)].
\end{equation*}
From \eqref{eqn:EMRM}, one has:
\begin{equation*}
   \mathbb E [S_n(2q,t,\chi)] \asymp \mathbb E [\Sigma_n(q,t,\chi)] \asymp 2^{-n(q-\psi(q)-1-\chi)}, 
\end{equation*}
so that 
\begin{equation*}
\mathbb E\Biggl[\left| \frac{S_n(2q,t,\chi) - a(2q) {\Sigma_n(q,t,\chi)}}{\mathbb E
    S_n(2q,t,\chi)} \right|^{1+\epsilon} \Biggr]\lesssim 2^{-n(- \psi[q(1+\epsilon)] +  (1+\epsilon) \psi(q)+ \epsilon(1+\chi))}.
\end{equation*}
Inequality \eqref{eqn:epsilonzero} and 
the Borel-Cantelli lemma end the proof.

\subsection{Proof of Proposition  \ref{lem:Approximation}}
\subsubsection{Outline of the proof} 

First note that if $t2^n$ is an integer, we have from
\eqref{eqn:EMRM}:
\begin{equation} \label{eqn:babar}
 \frac{\Sigma_n(q,t,0)}{\mathbb E [\Sigma_n(q,t,0)]} -
\frac{M_{2^{-n}}^{(q)} (t)}{\mathbb E [M_{2^{-n}}^{(q)} (t)]} = t^{-1} \biggl(
 2^{n(q-\psi(q)-1)} \frac{\Sigma_n(q,t,0)}{\gamma(q) T^{\psi(q)}} -
 M_{2^{-n}}^{(q)} (t) \biggr).
\end{equation}
We restrict ourselves to this case. Indeed, if $t2^n$ is not an
integer, we define $B_n(t)$ to be the same as $B_n(\lfloor t 2^n\rfloor 2^{-n})$ and
\[A_n(t) = A_n(\lfloor t 2^n\rfloor 2^{-n}) +\frac{M_{2^{-n}}^{(q)} (t)}{\mathbb E \bigl[ M_{2^{-n}}^{(q)} (t)\bigr]} -
  \frac{M_{2^{-n}}^{(q)} (\lfloor t 2^n\rfloor 2^{-n}) }{\mathbb E \bigl[
  M_{2^{-n}}^{(q)} (\lfloor t 2^n\rfloor 2^{-n})\bigr]}.\]
Since $\mathbb E \bigl[ M_{2^{-n}}^{(q)} (t)\bigr] = t$ for $t>0$, we
  clearly have that
\[\mathbb E \biggl[\Bigl| \frac{M_{2^{-n}}^{(q)} (t)}{\mathbb E \bigl[ M_{2^{-n}}^{(q)} (t)\bigr]} -
  \frac{M_{2^{-n}}^{(q)} (\lfloor t 2^n\rfloor 2^{-n}) }{\mathbb E \bigl[
  M_{2^{-n}}^{(q)} (\lfloor t 2^n\rfloor 2^{-n})\bigr]}
  \Bigr|\biggr]  \lesssim 2^{-n}.\]

Our proof relies on a  partition of the cones $\mathcal
A_l(u)$ that are used in the definition of the process
$\omega_l(u)$. We fix $\delta \in (0,1)$ and $\epsilon \in
(0,1)$. From Assumption $\mathbf A_q$, we can choose $\epsilon$ such
that $\mathbb E\bigl[M(t)^{q(1+\epsilon)}\bigr]<+\infty.$  For fixed $n$, $u$ in the dyadic interval
 $[k2^{-n}, (k\!+\!1)2^{-n})$ , $l \leq 2^{-n}$, and $m=\lfloor
 (1-\delta)n \rfloor$,
 we write:
\begin{equation*}
\mathcal A_l(u) =\tilde{ \mathcal  A}_{2^{-n}}(k) \cup \mathcal B_{l,2^{-n}}(u) \cup \mathcal T_{2^{-n}}(u)
\end{equation*}
where: 
 \begin{equation*}
  \tilde{ \mathcal  A}_{2^{-n}}(k) = \bigcap_{v \in [k2^{-n}, (k+1)2^{-n}]}
  \mathcal A_l(v),
 \end{equation*}
 $\mathcal T_{2^{-n}}(u)$ is the subset $\mathcal
A_l(u) \setminus \tilde{ \mathcal  A}_{2^{-n}}(k)$ that lies above the horizontal
line of y-coordinate $2^{-m}-2^{-n}$, and  $\mathcal B_{l,2^{-n}}(u)$
is the subset that lies below. Remark in particular that $\tilde{ \mathcal  A}_{2^{-n}}(k)$ does not depend on $l$ (see figure \ref{fig:partition}).

\setlength{\unitlength}{.7cm}

\begin{figure}[htbp] 

  \centering

  \begin{picture}(17,10)(-1,-1)


   \put(0,0){\line(1,0){16}}
   \put(0,0){\line(0,1){9}}

   \put(-.6,.4){\footnotesize $l$}
   \multiput(0,.5)(.45,0){13}{\line(1,0){.1}}
   \put(-1,1.8){\footnotesize $2^{-n}$}
   \multiput(0,2)(.45,0){13}{\line(1,0){.1}}
   \put(-2.2,5.9){\footnotesize $2^{-m} 
   \negthickspace -  2^{-n} $}
 
   \multiput(0,6)(.45,0){9}{\line(1,0){.1}}

   \put(6.7,-.8){\footnotesize $u$}
   \put(5.3,-.8){\footnotesize $k2^{-n}$}
   \put(7.5,-.8){\footnotesize $(k\!+\!1) 2^{-n}$}

   \multiput(6.75,0)(0,.45){2}{\line(0,1){.07}}
   \multiput(6,0)(0,.45){2}{\line(0,1){.07}}
   \multiput(8.05,0)(0,.45){2}{\line(0,1){.07}}

   \put(6.3,6){$\tilde{ \mathcal  A}_{2^{-n}}(k) $}

   \linethickness{.5mm}
   \ligne{11,1}{12.5,1}
   \ligne{11,2}{12.5,2}
   \ligne{11,1}{11,2}
   \ligne{12.5,1}{12.5,2}
   \linethickness{.1mm}
   \ligne{11,1.5}{12.5,1.5}

   \linethickness{.5mm}
   \ligne{11,3}{12.5,3}
   \ligne{11,4}{12.5,4}
   \ligne{11,3}{11,4}
   \ligne{12.5,3}{12.5,4}
   \linethickness{.1mm}
   \ligne{11.5,3}{11.5,4}
   \ligne{12,3}{12,4}

   \thinlines
   \put(12.8,1.3){ $\mathcal B_{l,2^{-n}}(u)$}
   \put(12.8,3.3){ $\mathcal T_{2^{-n}}(u)$}

   \linethickness{.3mm}
   \ligne{3.5,9}{7,2}
   \ligne{7,2}{10.5,9}

   \ligne{2.25,9}{6.5,.5}
   \ligne{6.5,.5}{7,.5}
   \ligne{7,.5}{11.25,9}
   \ligne{3.75,6}{5,6}
   \ligne{9,6}{9.75,6}
   \thinlines

   \linethickness{.1mm}
   \ligne{4,5.5}{5.25,5.5}
   \ligne{4.25,5}{5.5,5}
   \ligne{4.5,4.5}{5.75,4.5}
   \ligne{4.75,4}{6,4}
   \ligne{5,3.5}{6.25,3.5}
   \ligne{5.25,3}{6.5,3}
   \ligne{5.5,2.5}{6.75,2.5}

   \ligne{8.75,5.5}{9.5,5.5}
   \ligne{8.5,5}{9.25,5}
   \ligne{8.25,4.5}{9,4.5}
   \ligne{8,4}{8.75,4}
   \ligne{7.75,3.5}{8.5,3.5}
   \ligne{7.5,3}{8.25,3}
   \ligne{7.25,2.5}{8,2.5}

   \ligne{5.75,2}{7.75,2}
   \ligne{6,1.5}{7.5,1.5}
   \ligne{6.25,1}{7.25,1}

   \ligne{2.75,9}{2.75,8}
   \ligne{3.25,9}{3.25,7}
   \ligne{3.75,8.5}{3.75,6}
   \ligne{4.25,7.5}{4.25,6}
   \ligne{4.75,6.5}{4.75,6}

   \ligne{9.5,6}{9.5,7}
   \ligne{10,6.5}{10,8}
   \ligne{10.5,7.5}{10.5,9}
   \ligne{11,8.5}{11,9}

   \thinlines

   \linethickness{.1mm}
   \ligne{1.5,9}{5.75,.5}
   \ligne{5.75,.5}{6.25,.5}
   \ligne{6.25,.5}{10.5,9}

   \ligne{3.5,9}{7.75,.5}
   \ligne{7.75,.5}{8.25,.5}
   \ligne{8.25,.5}{12.5,9}

 \end{picture}
 \caption{\small Partition of $\mathcal A_l(u)$  } \label{fig:partition}
\end{figure}
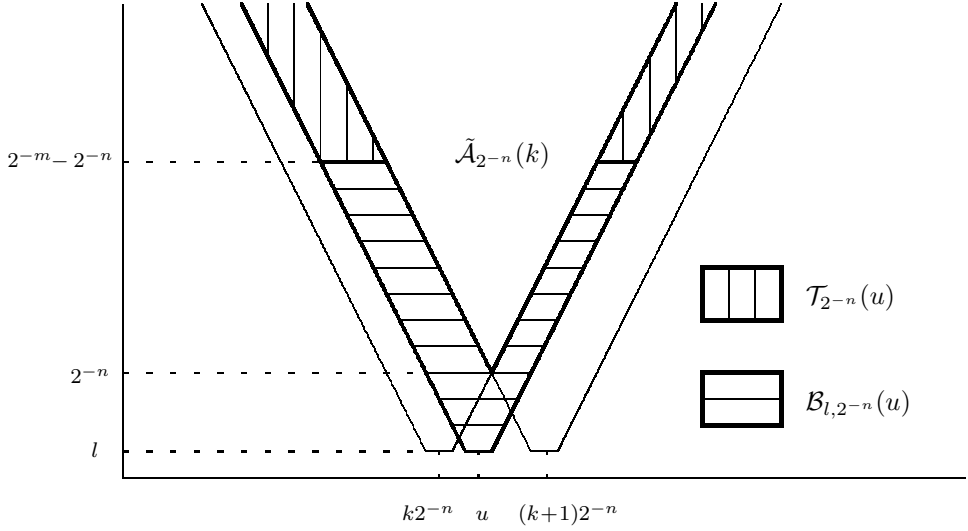

The images of these subsets by  $ P$ define new random processes:
 $\tilde \omega_{2^{-n}}(k) =  P
  (\tilde{\mathcal A}_{2^{-n}}(k))$,
  $\theta_{2^{-n}}(u) =  P
  (\mathcal T_{2^{-n}}(u))$ , and $\beta_{l,2^{-n}}(u) =  P (\mathcal
  B_{l,2^{-n}}(u))$. Likewise,   $\tilde \omega_{2^{-n}}^{(q)}(k)$, 
  $\theta^{(q)}_{2^{-n}}(u)$, and $\beta^{(q)}_{l,2^{-n}}(u)$ are defined by replacing $P$ by $P^{(q)}$.
It is straightforward to compute the surface of $\tilde{\mathcal A}_{2^{-n}}(k) $ 
 as measured by $\mu(dt,dl)=l^{-2}dt \otimes
dl$:

\begin{equation*} 
 \mu  (\tilde{\mathcal A}_{2^{-n}}(k)) = \log(T2^n)
\end{equation*} 
so that 
\begin{equation} \label{eqn:muAnc}
\mathbb E\bigl[e^{q \tilde \omega_{2^{-n}}(0)}\bigr] =  T^{\psi(q)} 2^{n\psi(q)}.
\end{equation}

We now justify this partition and  our approach. Let us
remark that from (\ref{eqn:muAnc}) and from the definition of the Laplace
exponents $\psi$ and  $\psi^{(q)}$, one has:
\begin{equation} \label{eqn:pivot} 
  T^{-\psi(q)} 2^{n(q-\psi(q)-1)}\left(2^{-n} e^{\tilde \omega _{2^{-n}}(k)}\right)^q =
2^{-n} e^{\tilde \omega ^{(q)} _{2^{-n}}(k)}.
\end{equation} 
This identity plays a key role in our proof. Indeed, we would like to justify the following approximations:
\[e^{q \omega_{2^{-n}}(k)} \approx e^{q\tilde \omega _{2^{-n}}(k)}\] and
\[\Sigma_n(q,t,0) \approx \sum_k  \bigl( 2^{-n} e^{\omega_{2^{-n}}(k)} \bigr)^q. \] If we were able to do this, we could probably as well justify the following:
\[e^{ \omega^{(q)}_{2^{-n}}(k)} \approx e^{\tilde \omega^{(q)} _{2^{-n}}(k)}\] and
\[ M_{2^{-n}}^{(q)}(t) \approx \sum_k   2^{-n} e^{\omega^{(q)}_{2^{-n}}(k)} . \] Renormalizing every quantity above by their respective expectations and using \eqref{eqn:pivot}, these approximations would thus provide a link between $\Sigma_n(q,t,0)$ and $ M_{2^{-n}}^{(q)}(t)$. We however have no easy method to justify these approximations; part
of the difficulty comes from the fact that the dependence between the variables which are
summed over $k$ decays very slowly.  Therefore we introduce the decomposition
\[ \mathcal A_l(u) \setminus \tilde{\mathcal A}_{2^{-n}}(k) = \mathcal B_{l,2^{-n}}(u) \cup
\mathcal T_{2^{-n}}(u)\]
in order to obtain independence for some of the  $ P(\mathcal
B_{l,2^{-n}}(u))$'s, which in particular enables us to apply
martingale inequalities. We find that some of the terms are more easily handled in $L^1$ norm, while the others are better handled in $L^{1+
\epsilon}$ norm, thus leading to the terms $A_n$ and $B_n$ in the statement of the proposition. 

Let us define 
\begin{equation} \label{eqn:defc}
c_{n,k} = \lim_{l \to 0}\int_{k2^{-n}}^{(k+1)2^{-n}}
  e^{\beta_{l,2^{-n}}(u)}du\end{equation} 
and
\begin{equation} \label{eqn:defd}
d_{n,k} = \lim_{l \to 0}\int_{k2^{-n}}^{(k+1)2^{-n}}
  e^{\beta_{l,2^{-n}}(u)+\theta_{2^{-n}}(u)}du,\end{equation} 
where, according to the results  of Bacry and Muzy in \cite{BM03}, the limits hold almost surely and in $L^1$  under $\mathbf A_1$. We will also need the term $\gamma_n(q)$:
\begin{equation}
\label{eqn:gamman}
 \gamma_n(q) =   2^{nq}\mathbb E \left[  c_{n,k}^q\right]  
\end{equation}
We will prove that 
 $\gamma_n(q) \rightarrow \gamma(q)$, where $\gamma(q)$  has been defined in
 \eqref{eqn:EMRM}. It can be seen directly that $\gamma(1) = \gamma_n(1)=1$.

We now write from \eqref{eqn:babar}:
\begin{equation*}
 \frac{\Sigma_n(q,t,0)}{\mathbb E [\Sigma_n(q,t,0)]} -
\frac{M_{2^{-n}}^{(q)} (t)}{\mathbb E [M_{2^{-n}}^{(q)} (t)]} = t^{-1}
(A_1+B_1+B_2+A_2)
\end{equation*}
with: 
\begin{equation*} 
A_1 = T^{-\psi(q)} 2^{n(q-\psi(q)-1)} \Biggl(
   {\gamma(q)^{-1}}\Sigma_n(q,t,0) 
    -  \sum_{k=0}^{ 2^nt -1}
    {\gamma_{n}(q)^{-1}} c_{n,k}^q   
   e^{q \tilde \omega_{2^{-n}}(k)} \Biggr)   ,
\end{equation*}
\begin{equation*} 
B_1 = T^{-\psi(q)} 2^{n(q-\psi(q)-1)}\Biggl(
  \sum_{k=0}^{ 2^nt -1}
    {\gamma_n(q)^{-1}} c_{n,k}^q   
   e^{q\tilde \omega_{2^{-n}}(k)} - (2^{-n} e^{\tilde \omega_{2^{-n}}(k)})^q\Biggr)
  ,
\end{equation*}
\begin{equation*}
B_2 =  \sum_{k=0}^{ 2^nt -1} \Biggl(
  2^{-n} e^{\tilde \omega^{(q)}_{2^{-n}}(k)} -  
    \int_{k2^{-n}}^{(k\!+\!1)2^{-n}}
    e^{ \beta^{(q)}_{2^{-n},2^{-n}}(u)+ \tilde \omega^{(q)}_{2^{-n}}(k)} du \Biggr) 
 ,
\end{equation*}
\begin{equation*} 
A_2 =   
  \sum_{k=0}^{ 2^nt -1}
    \int_{k2^{-n}}^{(k\!+\!1)2^{-n}}
  e^{ \beta^{(q)}_{2^{-n},2^{-n}}(u)+ \tilde \omega^{(q)}_{2^{-n}}(k)}  du  -
 M_{2^{-n}}^{(q)}(t)
\end{equation*}
 (recall from \eqref{eqn:pivot} that the difference $C$
\[ 
 C = T^{-\psi(q)} 2^{n(q-\psi(q)-1)} \sum_{k=0}^{ 2^nt -1} (2^{-n}
 e^{\tilde \omega_{2^{-n}}(k)})^q -\sum_{k=0}^{ 2^nt -1} 2^{-n} e^{\tilde \omega^{(q)}_{2^{-n}}(k)}
\]
is exactly zero.)
 The terms $A_n(t)$ and $B_n(t)$ in the statement of Proposition \ref{lem:Approximation} correspond to $A_1 + A_2
 $ and $B_1+B_2$. The properties stated in the proposition are easily verified,
 except for the upper bounds, which we prove below. The upper bound for $\mathbb E\bigl[|A_1|\bigr]$ and $\mathbb
 E\bigl[|A_2|\bigr]$ that we obtain
 is respectively $2^{-n\alpha_1}$ and $2^{-n\alpha_2}$ for some
 $\alpha_1, \, \alpha_2 >0$; this bound is established
 mainly from the fact
that $\theta_{2^{-n}}$ becomes zero when $m=\lfloor (1-\delta)n
 \rfloor \to +\infty$.
The
upper bound of $\mathbb E\bigl[|B_1|^{1+\epsilon}\bigr]$
and $\mathbb E\bigl[|B_2|^{1+\epsilon}\bigr]$ that we obtain is $2^{n(\psi^{(q)}
  (1+\epsilon)- \epsilon)}$; this bound is established as a consequence of the martingale inequality
\eqref{eqn:EvB}. 

The following technical Lemma  \ref{lem:Auxilliaire}
will be useful:
\begin{lemma} \label{lem:Auxilliaire}
Let $q>0$ and the infinitely divisible distribution $\pi(dx)$ be such
that Assumption $\mathbf A_q$ holds, so that we may choose  $\epsilon \in
(0,1)$  such that $\mathbb
E[M(t) ^{q(1+\epsilon)}] < +\infty$, and let $r$ be a real number in
$(0, q(1+\epsilon)]$. Then:\\
  {\bf (i)} Let $\mathcal C$ be a Borel set in $\mathbb R \times
  (0,+\infty)$ such that $\mu(\mathcal C) < +\infty$, and let $\mathcal C+s$ be
  the set $ \bigl\{(t,l) \
  \in \mathbb R \times
  (0,+\infty), \, (t-s,l) \in \mathcal C \bigr\}$ for $s \in \mathbb
  R$. Then for $t >0$ the moments 
  $\mathbb E\bigl[\sup_{0 \le u \le t} e^{r
    P(\mathcal C+u)} \bigr]$ and $\mathbb E\bigl[\sup_{0 \le u \le t} e^{
    (1+\epsilon)P^{(q)}(\mathcal C+u) }\bigr]$ are finite.\\
  {\bf (ii)} There exist $\alpha_1, \, \alpha_2>0$ such that \[ \mathbb E
  \Biggl[\sup_{u \in [0, 2^{-n}]} |1-e^{r \theta_{2^{-n}}(u)}| \Biggr]
 \lesssim 2^{-n \alpha_1} \text{ and } \mathbb E \Biggl[\sup_{u \in [0,
   2^{-n}]} |1-e^{ \theta^{(q)}_{2^{-n}}(u)} |\Biggr]
 \lesssim 2^{-n \alpha_2}.\]\\
  {\bf (iii)} This value $\alpha_1$ also satisfies \[ |\gamma_n(r)  -\gamma(r)| \lesssim 2^{-n \alpha_1}.\] \\
  {\bf (iv)} We have \[\mathbb E \bigl[|\gamma_{n}(q) - c_{n,k}^q
  |^{1+\epsilon}\bigr] \lesssim 1.\]
  {\bf (v)} We have
\[\mathbb E  \Biggl[\biggl| 1- 2^n\int_{0}^{2^{-n}}
    e^{\beta^{(q)}_{2^{-n},2^{-n}}(u)} du\biggr|^{1+\epsilon} \Biggr] \lesssim 1.\] 
\end{lemma}
The reader will find the proof of this lemma in the appendix.

\subsubsection{Upper bound for $\mathbb E\bigl[|A_1|\bigr]$}
Let us recall that 
\[\mathbb E [\Sigma_n(q,t,0)] \asymp 2^{-n(q-\psi(q)-1)}.\]
Then statement {\bf (iii)} of Lemma \ref{lem:Auxilliaire} shows that
\[2^{n(q-\psi(q)-1)} \Bigl|(\gamma(q)^{-1}-\gamma_n(q)^{-1}) \mathbb E
\bigl[\Sigma_n(t,q)\bigr]\Bigr| \lesssim 2^{-n \alpha_1}. \]
We therefore only have to give an upper bound for the expectation of: 
\[2^{n(q-\psi(q)-1)}\Bigl|
   \Sigma_n(q,t,0) - \sum_{k=0}^{ t2^{n} -1}
     c_{n,k}^q e^{q \tilde \omega_{2^{-n}}(k2^{-n})} \Bigr|. \]
We begin with the triangle inequality:
\begin{align*}
 \mathbb E \biggl[\Bigl|
   \Sigma_n(t,q) - \sum_{k=0}^{ t2^{n} -1}
     c_{n,k}^q e^{q \tilde \omega_{2^{-n}}(k2^{-n})}\Bigr| \biggr] 
&=  \mathbb E\biggl[  \Bigl|
\sum_{k=0}^{ t2^n -1} b_{n,k}^q- 
  c_{n,k} ^q e^{q \tilde \omega_{2^{-n}}(k2^{-n})}
  \Bigr|\biggr]\nonumber \\
 & \leq  t2^{n}  \, \mathbb E \biggl[
\Bigl|b_{n,0}^q  -c_{n,0} ^q  e^{q \tilde \omega_{2^{-n}}(0)}  \Bigr| \biggr]
\end{align*}
and
\[  \mathbb E \biggl[
\Bigl|b_{n,0}^q  -c_{n,0} ^q  e^{q \tilde \omega_{2^{-n}}(0)}  \Bigr| \biggr] =
\mathbb E \biggl[c_{n,0} ^q e^{q \tilde \omega_{2^{-n}}(0)}
\Bigl|\frac{b_{n,0}^q}{c_{n,0}^q e^{q \tilde \omega_{2^{-n}}(0)}}-1    \Bigr| \biggr] .\]
The term in the absolute value on the right is dominated by \[\sup_{u
  \in [0, 2^{-n}]} |e^{r \theta_{2^{-n}}(u)}-1|\] which is independent of
  $c_{n,0}$ and of $\tilde \omega_{2^{-n}}(0)$.
Moreover:
\begin{eqnarray*}
\mathbb E\bigl[c_{n,0}^q e^{q \tilde \omega_{2^{-n}}(0)}\bigr] & \asymp &  2^{n\psi(q)}\mathbb E \bigl[c_{n,0}^q\bigr] \nonumber\\
  & \asymp & 2^{n (q+\psi(q)) } \gamma_n(q)
\end{eqnarray*}
where (\ref{eqn:muAnc}) and (\ref{eqn:gamman}) have been used.
We finally use statement {\bf (ii)} of Lemma \ref{lem:Auxilliaire} to show that
\[\mathbb E \bigl[|A_1|\bigr] \lesssim 2^{-n \alpha_1} .\]

\subsubsection{Upper bound for $\mathbb E\bigl[|A_2|\bigr]$}
The proof here is very similar to the previous one. We write 
\[M_{2^{-n}}^{(q)}(t)=\sum_{k=0}^{t2^n-1}
M_{2^{-n}}^{(q)}\bigl((k+1)2^{-n}\bigr)-M_{2^{-n}}^{(q)}\bigl(k2^{-n}\bigr)\]
and apply the triangle inequality:
\[\mathbb E\bigl[|A_2|\bigr] \lesssim 2^{n}\mathbb
E\Biggl[\biggl|\int_{0}^{2^{-n}}e^{\beta^{(q)}_{2^{-n},2^{-n}}(u)+
    \tilde \omega^{(q)}_{2^{-n}}(0)} du  - M_{2^{-n}}^{(q)}(2^{-n})\biggr|\Biggr] .\]
Using the same arguments as in the previous section, we have:
\[\mathbb E\bigl[|A_2|\bigr] \lesssim 2^{n} \mathbb E \Biggl[
  e^{\tilde \omega^{(q)}_{2^{-n}}(0)}
  \int_{0}^{2^{-n}}  e^{\beta^{(q)}_{2^{-n},2^{-n}}(u)}du  
 \sup_{u \in [0,
   2^{-n}]}  |1-e^{\theta^{(q)}_{2^{-n}}(u)}|
\Biggr].
\]
Each of the three terms in the expectation of the right hand side is
independent of the other two. Moreover, the expectation of the
exponential term is 1, and the expectation of the integral term is
$2^{-n}$. Applying   {\bf (ii)} of Lemma \ref{lem:Auxilliaire} gives
the result.

\subsubsection{Upper bound for $\mathbb E\bigl[|B_1|^{1+\epsilon}\bigr]$}
We write  $$Z_{k2^{-n}} = e^{q \tilde \omega_{2^{-n}}(k2^{-n})} \bigl( \gamma_n(q)-c_{n,k}^q \bigr),$$
so that 
\begin{equation} \label{eqn:B1} B_1 \asymp 2^{-n(\psi(q)+1)}
  \sum_{k=1}^{  t2^n }Z_{k2^{-n}}.
\end{equation}
We first apply the convexity inequality: \[  \left| \sum_{k=1}^N
x_k \right|^{1+\epsilon}  \leq N^{\epsilon} \sum_{k=1}^N |x_k|^{1+\epsilon} .\] 
This gives:
\begin{eqnarray*}
 \mathbb E \biggl[ \Bigl| \sum_{k=0}^{  t 2^n-1 }
 Z_{k2^{-n}}\Bigr|^{1+\epsilon} \biggr]
 & = & \mathbb E \biggl[ \Bigl| \sum_{i=0}^{  t2^m-1} \sum_{j=0}^{2^{n-m}-1}
Z_{i2^{-m} +
  j2^{-n}}\Bigr|^{1+\epsilon} \biggr]\\
& \leq & 2^{\epsilon (n-m)} \sum_{j=0}^{2^{n-m}-1} \mathbb E \biggl[ \Bigl|
\sum_{i=0}^{  t2^m-1}
 Z_{i2^{-m} +
  j2^{-n}}\Bigr|^{1+\epsilon} \biggr].
\end{eqnarray*}
  From the stationarity of the  $Z_{k2^{-n}}$'s,  $\mathbb E \Bigl[ \bigl|
\sum_{i=0}^{  t2^m-1}
 Z_{i2^{-m} +
  j2^{-n}}\bigr|^{1+\epsilon} \Bigr]$ does not depend on $j$. We now show
that inequality (\ref{eqn:EvB}) can be applied to this term. For
  $j=0$  and $\bar \imath \leq t2^{m}-1$, one has:

\begin{equation*}
\mathbb E \Bigl[  Z_{\bar  \imath 2^{-m} } \bigl| \sum_{i=0}^{\bar \imath -1} Z_{i2^{-m} } \bigr. \Bigr]  =
\mathbb E \biggl[ \mathbb E \Bigl[ \bigl. Z_{\bar \imath 2^{-m}} \big| e^{q
\tilde \omega_{2^{-n}}(\bar \imath 2^{-m})},\sum_{i=0}^{\bar \imath -1} Z_{i2^{-m} }
 \Bigr] \Bigr| \sum_{i=0}^{\bar \imath -1} Z_{i2^{-m} } 
\biggr] .
\end{equation*}

By factorizing, the term $ \mathbb E \Bigl[ Z_{\bar \imath 2^{-m}}
  \big| e^{q \tilde
\omega_{2^{-n}}(\bar \imath 2^{-m})},\sum_{i=0}^{\bar \imath -1} Z_{i2^{-m} }
 \Bigr]$ becomes:
\begin{equation*}
  e^{q \tilde \omega_{2^{-n}}(\bar \imath 2^{-m})}
\mathbb E \Bigl[ \gamma_n(q)- ( 2^n c_{n, \bar \imath 2^{-m+n}})^q
  \big| e^{q \tilde \omega_{2^{-n}}(\bar \imath 2^{-m})},\sum_{i=0}^{\bar
\imath -1} Z_{i2^{-m} }  \Bigr] .
\end{equation*}
Let us now recall that  $\beta_{l,2^{-n}}(u) =
 P(\mathcal B_{l,2^{-n}}(u))$ and observe that if $u$ lies between 
$\bar \imath 2^{-m}$ and $\bar \imath 2^{-m} + 2^{-n}$, then $\mathcal
B_{l,2^{-n}}(u)$ is of empty intersection with
 $\tilde{\mathcal A}_{2^{-n}}(\bar \imath 2^{-m})$,
  the $\tilde{\mathcal A}_{2^{-n}}( i 2^{-m})$'s for $i \leq \bar \imath -1$, and the $\mathcal B_{l,2^{-n}}(v)$'s for $v \leq (\bar \imath -1) 2^{-m} +
 2^{-n}$.

\setlength{\unitlength}{.6cm}
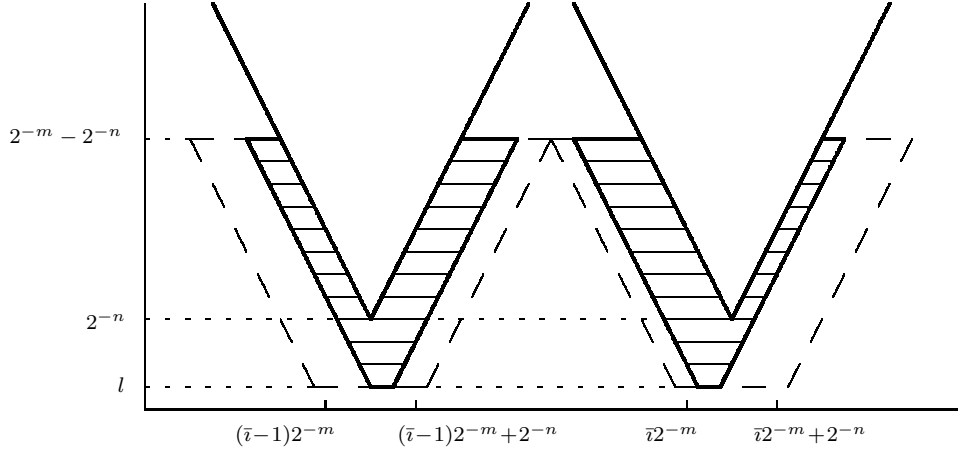
\begin{figure}[htbp]
  \centering
   \begin{picture}(20,11)(-2,-2)


   \put(0,0){\line(1,0){18}}
   \put(0,0){\line(0,1){9}}

   \put(4,0){\line(0,1){.2}}
   \put(6,0){\line(0,1){.2}}
   \put(12,0){\line(0,1){.2}}
   \put(14,0){\line(0,1){.2}}

   \put(-.6,.4){\footnotesize $l$}
   \multiput(0,.5)(.45,0){9}{\line(1,0){.1}}
   \multiput(6,.5)(.45,0){15}{\line(1,0){.1}}
   \put(-1.3,1.8){\footnotesize $2^{-n}$}
   \multiput(0,2)(.45,0){10}{\line(1,0){.1}}
   \multiput(6.5,2)(.45,0){11}{\line(1,0){.1}}
   \put(-3,5.9){\footnotesize $2^{-m}-2^{-n}$}
   \multiput(0,6)(.45,0){3}{\line(1,0){.1}}

   \put(2,-.7){\footnotesize $(\bar \imath \! - \! 1) 2^{-m}$}
   \put(5.6,-.7){\footnotesize $(\bar \imath \! - \! 1) 2^{-m} \! + \! 2^{-n} $}
   \put(11.1,-.7){\footnotesize $\bar \imath 2^{-m}$}
   \put(13.5,-.7){\footnotesize $\bar \imath 2^{-m} \! + \! 2^{-n} $}

   \linethickness{.3mm}
   \ligne{1.5,9}{5,2}
   \ligne{5,2}{8.5,9}

   \ligne{2.25,6}{5,.5}
   \ligne{5,.5}{5.5,.5}
   \ligne{5.5,.5}{8.25,6}
   \ligne{2.25,6}{3,6}
   \ligne{7,6}{8.25,6}

   \linethickness{.1mm}
   \ligne{2.5,5.5}{3.25,5.5}
   \ligne{2.75,5}{3.5,5}
   \ligne{3,4.5}{3.75,4.5}
   \ligne{3.25,4}{4,4}
   \ligne{3.5,3.5}{4.25,3.5}
   \ligne{3.75,3}{4.5,3}
   \ligne{4,2.5}{4.75,2.5}

   \ligne{6.75,5.5}{8,5.5}
   \ligne{6.5,5}{7.75,5}
   \ligne{6.25,4.5}{7.5,4.5}
   \ligne{6,4}{7.25,4}
   \ligne{5.75,3.5}{7,3.5}
   \ligne{5.5,3}{6.75,3}
   \ligne{5.25,2.5}{6.5,2.5}

   \ligne{4.25,2}{6.25,2}
   \ligne{4.5,1.5}{6,1.5}
   \ligne{4.75,1}{5.75,1}
   \thinlines

   \thinlines

   \linethickness{.3mm}
   \ligne{9.5,9}{13,2}
   \ligne{13,2}{16.5,9}

   \ligne{9.5,6}{12.25,.5}
   \ligne{12.25,.5}{12.75,.5}
   \ligne{12.75,.5}{15.5,6}
   \ligne{9.5,6}{11,6}
   \ligne{15,6}{15.5,6}

   \linethickness{.1mm}
   \ligne{9.75,5.5}{11.25,5.5}
   \ligne{10,5}{11.5,5}
   \ligne{10.25,4.5}{11.75,4.5}
   \ligne{10.5,4}{12,4}
   \ligne{10.75,3.5}{12.25,3.5}
   \ligne{11,3}{12.5,3}
   \ligne{11.25,2.5}{12.75,2.5}

   \ligne{14.75,5.5}{15.25,5.5}
   \ligne{14.5,5}{15,5}
   \ligne{14.25,4.5}{14.75,4.5}
   \ligne{14,4}{14.5,4}
   \ligne{13.75,3.5}{14.25,3.5}
   \ligne{13.5,3}{14,3}
   \ligne{13.25,2.5}{13.75,2.5}

   \ligne{11.5,2}{13.5,2}
   \ligne{11.75,1.5}{13.25,1.5}
   \ligne{12,1}{13,1}
   \thinlines

 \linethickness{.1mm}
 \ligne{1,6}{1.25,5.5}
 \ligne{1.5,5}{1.75,4.5}
 \ligne{2,4}{2.25,3.5}
 \ligne{2.5,3}{2.75,2.5}
 \ligne{3,2}{3.25,1.5}
 \ligne{3.5,1}{3.75,.5}

 \ligne{3.75,.5}{4.25,.5}
 \ligne{4.75,.5}{5.25,.5}
 \ligne{5.25,.5}{5.75,.5}
 \ligne{5.75,.5}{6.25,.5}

 \ligne{6.25,.5}{6.5,1}
 \ligne{6.75,1.5}{7,2}
 \ligne{7.25,2.5}{7.5,3}
 \ligne{7.75,3.5}{8,4}
 \ligne{8.25,4.5}{8.5,5}
 \ligne{8.75,5.5}{9,6}

 \ligne{1,6}{1.5,6}
 \ligne{2,6}{2.5,6}

 \ligne{7,6}{7.5,6}

 \ligne{8.5,6}{9,6}

 \ligne{9,6}{9.25,5.5}
 \ligne{9.5,5}{9.75,4.5}
 \ligne{10,4}{10.25,3.5}
 \ligne{10.5,3}{10.75,2.5}
 \ligne{11,2}{11.25,1.5}
 \ligne{11.5,1}{11.75,.5}

 \ligne{11.75,.5}{12.25,.5}

 \ligne{13.25,.5}{13.75,.5}

 \ligne{14.25,.5}{14.5,1}
 \ligne{14.75,1.5}{15,2}
 \ligne{15.25,2.5}{15.5,3}
 \ligne{15.75,3.5}{16,4}
 \ligne{16.25,4.5}{16.5,5}
 \ligne{16.75,5.5}{17,6}

 \ligne{9,6}{9.5,6}
 \ligne{10,6}{10.5,6}

 \ligne{15,6}{15.5,6}
 \ligne{16,6}{16.5,6}

 \end{picture}
 \caption{\small  $\mathcal B_{l,2^{-n}}(u) \cap \mathcal B_{l,2^{-n}}(u')$ is
 empty if $|u-u'| > 2^{-m} - 2^{-n}$.}
\vspace{1cm}
\end{figure}

The random variables generated by $ P$ and these subsets of
the halfplane are therefore independent, so that the conditional
expectation above is non-random, and even zero from the definition
 (\ref{eqn:gamman}) of $\gamma_n(q)$. Then $\bigl( \sum_{i=0}^{\bar \imath } Z_{i2^{-m}
} \, 0 \le \bar \imath \leq   t2^m \rfloor -1 \bigr)$  is indeed a sequence of martingale increments,
and inequality (\ref{eqn:EvB}) applies:
\begin{equation*}
 \mathbb E \biggl[\bigl| \sum_{i=0}^{  t2^m  -1}
 Z_{i2^{-m}}\Bigr|^{1+\epsilon}\biggr] \lesssim   \sum_{i=0}^{ 
 t2^m  -1}
 \mathbb E\biggl[\bigl|   Z_{i2^{-m} } \Bigr|^{1+\epsilon}\biggr].
\end{equation*}
Going back to  $B_1$, we obtain:
\begin{equation*}
\mathbb E  \bigl[|B_1|^{1+\epsilon}\bigr]
\lesssim 2^{-n(1+\epsilon )(1+\psi(q))} 2^{(1+\epsilon)(n-m)}
2^{m}\mathbb E \bigl[|Z_k|^{1+\epsilon}\bigr].
\end{equation*}
Let us now give orders of magnitude for $E \bigl[|Z_k|^{1+\epsilon}\bigr]$:
\begin{eqnarray*}
 E \bigl[|Z_0|^{1+\epsilon}\bigr] & \asymp & \mathbb E\bigl[ e^{q(1+\epsilon)
 \tilde \omega_{2^{-n}} (0)}\bigr] \text{ (by \textbf{(iv)} of Lemma \ref{lem:Auxilliaire})}\\
 & \asymp & 2^{n \psi(q(1+\epsilon))} \text{ (by \eqref{eqn:muAnc})}.
\end{eqnarray*} We defined  $m = \lfloor(1-\delta )n\rfloor.$  Hence
\begin{equation*}
\mathbb E  \bigl[|B_1|^{1+\epsilon} \bigr]
\lesssim 2^{-\lfloor (1-\delta )n\rfloor \epsilon + n\psi(q(1+\epsilon)) -
  n(1+\epsilon )\psi(q)}.
\end{equation*}
As $\delta$  can be chosen arbitrarily small, the result follows.

\subsubsection{Upper bound for  $\mathbb E\bigl[|B_2|^{1+\epsilon}\bigr]$}
We now
write:
\[Z_k = e^{\tilde \omega^{(q)}_{2^{-n}}(k2^{-n})} \biggl( 1- 2^{n}\int_{k2^{-n}}^{(k\!+\!1)2^{-n}}
    e^{\beta^{(q)}_{2^{-n},2^{-n}}(u)} du\biggr)\]
so that
\[B_2 = 2^{-n}\sum_{k=0}^{t2^{n}-1} Z_k.\]
Going along the same lines as the previous section, we find:
\[\mathbb E \bigl[|B_2|^{1+\epsilon}\bigr] \lesssim 2^{-m\epsilon}
\mathbb E\bigl[|Z_0|^{1+\epsilon}\bigr].\]
From \eqref{eqn:muAnc} and  \textbf{(v)} of Lemma
\ref{lem:Auxilliaire}, we have 
\[\mathbb E\bigl[|Z_0|^{1+\epsilon}\bigr] \lesssim
2^{n\psi^{(q)}(1+\epsilon)}.\]
Letting $\delta \to 0$ achieves the proof.

\section{Proof of Theorem \ref{thm:CvMixte}} \label{sec:dem2}

Note that Proposition \ref{lem:MRW} shows that  if
Theorem \ref{thm:CvMixte}
holds for $\Sigma_n$, then it holds for $S_n$. 
In order to show that the theorem holds for $\Sigma_n$, we proceed in
two steps. First we  show that one can use Proposition
\ref{lem:Approximation} so as to bound 
\[\mathbb E \Biggl[\left| \frac{\Sigma_n
    (q,t,\chi)}{\mathbb E [\Sigma_n (q,t,\chi)]} - 1 \right|\Biggr] \] by
the sum of a term that goes to zero exponentially fast and the
quantity  $2^{-n\chi\epsilon} \mathbb E   \bigl[| M_{2^{-n}}^{(q)}(T)-
  T|^{1+\epsilon}\bigr]$.
 Then the proof is over if Assumption $\mathbf B^{(q)}(0)$ holds,
 since in this case  $M_{2^{-n}}^{(q)}(T)$ 
converges in $L^{1+\epsilon}$ (see Proposition \ref{thm:BM03moments}). However, in the case
where only  $\mathbf B^{(q)}(\chi)$ holds, a  bit more work is required to show that
this quantity indeed goes to zero exponentially fast: this is our
second step. We finally apply the Borel-Cantelli lemma to obtain almost sure convergence.

\subsection{First step}
Let us define: 
\[J = J(t,n,\chi,T) =  \lfloor 2^{n\chi}t/T\rfloor -1.\]
Then for $0 \le j \le J - 1$, we set
\[\Delta_n(j)  =  \frac{\Sigma_n\bigl(q,(j\!+\!1)T,0) -
  \Sigma_n\bigl(q,jT,0\bigr)}{\mathbb E \bigl[
  \Sigma_n\bigl(q,T,0\bigr)\bigr]} -1,\]
and 
\[\Delta_n(J)  =  \frac{\Sigma_n\bigl(q,t,\chi) -
  \Sigma_n\bigl(q,JT,0\bigr)}{\mathbb E \bigl[
  \Sigma_n\bigl(q,T,0\bigr)\bigr]} -1,\]
so that from \eqref{eqn:EMRM}
\[ \frac{\Sigma_n
    (q,t,\chi)}{\mathbb E [\Sigma_n (q,t,\chi)]} - 1 \asymp
    2^{-n\chi}  \sum_{j=0}^{J}
    \Delta_n(j) .\]
Note that 
\[0 \le \mathbb E\bigl[ \Sigma_n\bigl(q,t,\chi) -
  \Sigma_n\bigl(q,JT,0\bigr) \bigr] \le \mathbb E
  \bigr[\Sigma_n\bigl(q,T,0\bigr)\bigr],\]
so that $\Delta_n(J)$ is bounded in $L^1$. Therefore \[2^{-n\chi}\mathbb E
  \bigl[|\Delta_n(J)|\bigr] \lesssim 2^{-n\chi}.\]
We now examine upper bounds for 
\[ \mathbb E \biggl[2^{-n\chi} \Bigl| \sum_{j=0}^{J-1}
    \Delta_n(j) \Bigr| \biggr].\]
We introduce the process $M_{2^{-n}}^{(q)}$ (let us recall that $\mathbb
    E[M_{2^{-n}}^{(q)}(T)]=T $). For $0 \le j \le J-1$
\begin{equation*}\begin{split}
\Delta_n(j)  = &  \frac{\Sigma_n\bigl(q,(j\!+\!1)T,0) -
  \Sigma_n\bigl(q,jT,0)\bigr)}{\mathbb E \bigl[
  \Sigma_n\bigl(q,T,0)\bigr)\bigr]} 
   - \frac{ M_{2^{-n}}^{(q)}\bigl((j\!+\!1)T\bigr)-M_{2^{-n}}^{(q)}\bigl(jT\bigr)}{T} \\
 & +  \frac{
  M_{2^{-n}}^{(q)}\bigl((j\!+\!1)T\bigr)-M_{2^{-n}}^{(q)}\bigl(jT\bigr)}{T} -1 ,
\end{split} \end{equation*} 
From this, we write  $\Delta_n(j) \asymp \Delta_{n,1}(j)+
\Delta_{n,2}(j)$ with
\[\Delta_{n,1}(j) = A_n\bigl((j+1)T\bigr) - A_n(jT)\] and
\begin{equation*}
\Delta_{n,2}(j) = B_n\bigl((j+1)T\bigr) - B_n(jT) + \frac{
  M_{2^{-n}}^{(q)}\bigl((j\!+\!1)T\bigr)-M_{2^{-n}}^{(q)}\bigl(jT\bigr)}{T} -1,
\end{equation*}
where the terms $A_n$ and $B_n$ have been introduced in Proposition \ref{lem:Approximation}.
Thus, 
\begin{align*}
 \mathbb E \biggl[ 2^{-n\chi}\Bigl| \sum_{j=0}^{J-1}
    \Delta_n(j) \Bigr| \biggl]
& \le \mathbb E \biggl[ 2^{-n\chi} \Bigl| \sum_{j=0}^{J-1}
     \Delta_{n,1}(j)\Bigr|  \biggr] + \mathbb E \biggl[
    2^{-n\chi} \Bigl| \sum_{j=0}^{J-1} 
    \Delta_{n,2}(j)\Bigr|  \biggr].
\end{align*}
The triangle inequality shows that 
\[\mathbb E \biggl[ 2^{-n\chi} \Bigl| \sum_{j=0}^{J-1}
     \Delta_{n,1}(j)\Bigr|  \biggr] \lesssim \mathbb
    E\bigl[\Delta_{n,1}(0)\bigr] = \mathbb E\bigl[|A_n(T)|\bigr].\]
According to Proposition \ref{lem:Approximation}, this term goes
    exponentially fast to zero.
Let us now deal with the terms $\Delta_{n,2}(j)$.  From Assumptions
    $\mathbf A_q$ and
    $\mathbf B^{(q)}(\chi)$, we may choose $\epsilon \in (0,1)$ such
    that $\mathbb E \bigl[M(t)^{q(1+\epsilon)}\bigr]<+\infty$ and $\psi^{(q)}(1+\epsilon) -
 \epsilon - \chi \epsilon < 0$. For this $\epsilon$, we have:
\begin{equation} \label{eqn:stepone} \mathbb E \biggl[
    2^{-n\chi} \Bigl| \sum_{j=0}^{J-1} 
    \Delta_{n,2}(j)\Bigr|  \biggr] \le \Biggl( \mathbb E \biggl[
    2^{-n\chi(1+\epsilon)} \Bigl| \sum_{j=0}^{J-1} 
    \Delta_{n,2}(j)\Bigr|^{1+\epsilon}
    \biggr]\Biggr)^{1/(1+\epsilon)}.\end{equation}
Moreover,
\begin{align*}
\mathbb E \biggl[ \Bigl| \sum_{j=0}^{J-1} 
    \Delta_{n,2}(j)\Bigr|^{1+\epsilon}
    \biggr]  & =  \mathbb E \biggl[ \Bigl| \sum_{j=0}^{ \lfloor (J-1)/2\rfloor} 
    \Delta_{n,2}(2j) + \sum_{j=0}^{\lfloor  J/2\rfloor-1} 
    \Delta_{n,2}(2j+1) \Bigr|^{1+\epsilon} \biggr]\\
 & \lesssim \mathbb E \biggl[ \Bigl| \sum_{j=0}^{ \lfloor (J-1)/2\rfloor} 
    \Delta_{n,2}(2j) \Bigr|^{1+\epsilon} \biggr].
\end{align*}
Since the increments of $B_n$ and $M_{2^{-n}}^{(q)}$ are stationary and independent
as soon as
they are taken on intervals that lie at a distance larger than
$T$, the $ \Delta_{n,2}(2j)$'s are i.i.d. random variables. From
Proposition \ref{lem:Approximation}, $\mathbb E[B_n(T)] = 0$, so that
these variables are
also centered.  Therefore, inequality (\ref{eqn:EvB}) can be applied, which gives:
\begin{equation*}
\mathbb E \biggl[ \Bigl| \sum_{j=0}^{ \lfloor (J-1)/2\rfloor} 
    \Delta_{n,2}(2j) \Bigr|^{1+\epsilon} \biggr] \lesssim 2^{n\chi}
    \mathbb E\bigl[|\Delta_{n,2}(0)|^{1+\epsilon}].
\end{equation*}
From the definition of $\Delta_{n,2}(0)$, we have:
\[\mathbb E\bigl[|\Delta_{n,2}(0)|^{1+\epsilon}\bigr] \lesssim \mathbb
E\bigl[|B_n(T)|^{1+\epsilon}\bigr] + \mathbb E \bigl[|M_{2^{-n}}^{(q)}(T)-T|^{1+\epsilon}\bigr]
\] and from the upper bound for $\mathbb
E\bigl[|B_n(T)|^{1+\epsilon}\bigr]$ in Proposition  \ref{lem:Approximation}, we have
\[\mathbb E \biggl[
    2^{-n\chi(1+\epsilon)} \Bigl|
    \sum_{j=0}^{J-1} 
    \Delta_{n,2}(j)\Bigr|^{1+\epsilon}
    \biggr] \lesssim 2^{n(\psi^{(q)}(\epsilon) - \epsilon - \chi
    \epsilon)} + 2^{-n\chi\epsilon} \mathbb E
    \bigl[|M_{2^{-n}}^{(q)}(T)-T|^{1+\epsilon}\bigr].\]
Recall that we have chosen $\epsilon$ so that  $\psi^{(q)}(\epsilon) - \epsilon - \chi
    \epsilon < 0$. Going back to \eqref{eqn:stepone}, we have
    therefore proved that there exists some $\xi>0$ such that
\[ \mathbb E \biggl[ \Bigl| \frac{\Sigma_n
    (q,t,\chi)}{\mathbb E [\Sigma_n (q,t,\chi)]} - 1 \bigr| \biggr]  \lesssim 
 2^{-n\xi} + 2^{-n\chi\epsilon} \mathbb E
    \bigl[|M_{2^{-n}}^{(q)}(T)-T|^{1+\epsilon}\bigr].\]

\subsection{Second step} We show here that $2^{-n\chi\epsilon} \mathbb E   \bigl[| M_{2^{-n}}^{(q)}(T)-
  T|^{1+\epsilon}\bigr]$ goes to zero exponentially fast. It will be enough to show that 
  $2^{-n\chi\epsilon} \mathbb E[M_{2^{-n}}^{(q)}(T)^{1+\epsilon}]$ goes to
  zero exponentially fast.

We define $\omega_{2^{-n}}^{(q,\epsilon)} = \omega_{2^{-n}}^{(q)} - n \log(2) \chi 
 \epsilon / (1+\epsilon)$ and $M_{2^{-n}}^{(q,\epsilon)}(t) = \int_0^t e^{\omega_{2^{-n}}^{(q,\epsilon)} (u)}du$
so that
\begin{equation*}
 2^{-n\chi\epsilon} \mathbb E\Bigl[{M_{2^{-n}}^{(q)}(T)}^{1+\epsilon}\Bigr]
 =  \mathbb E\Bigl[{M_{2^{-n}}^{(q,\epsilon)}(T)}^{1+\epsilon}\Bigr].
\end{equation*}
We now give two lemmas that are directly inspired by the proofs used by
Bacry and Muzy in \cite{BM03}.

\begin{lemma} \label{lem:MajorationEMq}
Under Assumption $\mathbf A_1$, for $n_0 \in \mathbb N$ such that $2^{-n} \leq T 2^{-n_0}$,
\begin{equation*}
 \mathbb E\Bigl[{M_{2^{-n}}^{(q,\epsilon)}(T)}^{1+\epsilon}\Bigr] \leq 2^{\epsilon+n_0}
 \mathbb E\Bigl[ {M_{2^{-n}}^{(q,\epsilon)}(2^{-n_0}T)}^{1+\epsilon} \Bigr] +
 c 2^{-n\chi\epsilon}
\end{equation*}
where   $c>0$ depends on $n_0$ but not on $n$.
\end{lemma}

\begin{lemma} \label{lem:Invariance}
Under Assumption $\mathbf A_1$, for $\lambda \in (0,1)$,  $l \in (0,T]$, and $t \in (0,T]$,
\[M_{\lambda l}^{(q)}(\lambda t) \stackrel{d}{=} \lambda e^{\Omega_\lambda}M_{
  l}^{(q)}( t)\] where $\Omega_\lambda$ is an infinitely divisible
  random variable independent of $M^{(q)}$ that satisfies
\[\mathbb E[e^{(1+\epsilon)\Omega_\lambda }] = \lambda^{-\psi^{(q)}(1+\epsilon)}.\]
\end{lemma}
We give a proof for lemma \ref{lem:MajorationEMq} in the
 appendix. Lemma \ref{lem:Invariance} is  a less general statement of Lemma 2
 of Bacry and Muzy in \cite{BM03}. We do not reproduce its proof;
 it involves the computation of the characteristic function of the random vector $\bigl(\omega_l(t_1),\dots,\omega_l(t_k)
 \bigr)$ through some elaborate combinatorial arguments. 

 Let us now remark that:
\begin{align*}
 \mathbb E\Bigl[{M_{2^{-n}}^{(q,\epsilon)}(2^{-n_0}T)}^{1+\epsilon}\Bigr] & =
 2^{-n\chi\epsilon} \mathbb E\Bigl[{M_{2^{-n}}^{(q)}(2^{-n_0}T)}^{1+\epsilon}\Bigr]  \\
& =  2^{-n\chi\epsilon-n_0(1+\epsilon) + n_0 \psi^{(q)} (1+\epsilon)}
 \mathbb E\Bigl[{M_{2^{-n+n_0}}^{(q)}(T)}^{1+\epsilon}\Bigr] \\
&= 2^{-n_0[1+(1+\chi)\epsilon-\psi^{(q)} (1+\epsilon)]} \mathbb
 E\Bigl[{M_{2^{-n+n_0}}^{(q,\epsilon)}(T)}^{1+\epsilon}\Bigr] ,
\end{align*}
where we used Lemma \ref{lem:Invariance}. Then from Lemma
\ref{lem:MajorationEMq} we see that: 
\begin{equation*}
 x_n \leq a \, x_{n-n_0} + c \, b^{n} 
\end{equation*}
with 
\[x_n =  2^{-n\chi\epsilon} \mathbb E\Bigl[{M_{2^{-n}}^{(q)}(T)}^{1+\epsilon}\Bigr],\]
\[ a =  2^{\epsilon -n_0[(1+\chi)\epsilon-\psi^{(q)} (1+\epsilon)]}\]
and 
\[ b  = 2^{- \chi \epsilon}.\]
From this we deduce by induction that
\[x_n \lesssim (n-n_0) \left[\max (a,b)\right]^{n-\,n_0}. 
\]
For a fixed $n_0$ large enough, we will have  $0<a<1$ since $\epsilon$
has been chosen such that \[(1+\chi)\epsilon-\psi^{(q)} (1+\epsilon)
>0.\]  
This achieves the proof.

\section{Proof of Theorem \ref{thm:CvZeta}} \label{sec:dem3}
We follow closely the proof that is given by Ossiander and Waymire in \cite{OW00}
or Bacry \emph{et al.} in \cite{BGHM} concerning Mandelbrot
cascades. We reproduce it here for the sake of
completeness.  Note that it follows exactly the same pattern for $S_{n}$ or $\Sigma_{n}$.

The case $0<q< q_\chi$ is a direct consequence of Theorems
\ref{thm:CvS} and \ref{thm:CvMixte} and of the relation:
\begin{equation*} 
\mathbb E \bigl[S_{n} (2q,t,\chi)\bigr] \asymp \mathbb E \bigl[\Sigma_{n} (q;t,\chi)\bigr] \asymp 
2^{-n[q-\psi(q)-1- \chi]}.
\end{equation*}

We now consider the case $q \geq q_\chi$. First notice that since we
have from
Theorems \ref{thm:CvS} and \ref{thm:CvMixte} that almost surely \[ \frac{\log_2(\Sigma_{n}(q,t,\chi))}{-n} \to q-\psi(q)-1 - \chi \quad\text{as}\quad n \to +\infty\] for $q \in (0,q_\chi)$, we may assume that
with probability one the convergence occurs for all values of $q$ in a
dense subset of this interval. Let us now take $q \geq q_\chi$.

Let $\rho$ be in $ (0,1)$. From the sub-additivity of $x \mapsto
x^{\rho}$,
\[ \Sigma_{n}(q,t,\chi)^\rho \leq  \Sigma_{n}(\rho q,t,\chi)\]
so that
\[\liminf_{n \to +\infty} \frac{\log_2 (\Sigma_{n}(q,t,\chi))}{-n} \geq \liminf_{n \to +\infty}
\frac{\log_2 (\Sigma_{n}(\rho q,t,\chi))}{-\rho n}.\]
Letting $\rho \to q_\chi/q$, we obtain that
\[\liminf_{n \to +\infty} \frac{\log_2 (\Sigma_{n}(q,t,\chi))}{-n} \geq q
\frac{q_\chi-\psi(q_\chi)-1-\chi}{q_\chi}.\]
It is then easily checked from the definition of $q_\chi$ that the
right hand side is equal to $q (1-\psi'(q_\chi))$.

Next, let us choose $q_1$, $q_2$ so that $0<q_1<q_2<q_\chi$. Then
\begin{align*}
\Sigma_{n}(q_2,t,\chi) & \leq \sup_{0 \leq k  \leq  \lfloor t 2^{n}\rfloor-1}
b_{n,k}^{q_2-q_1} \Sigma_{n}(q_1,t,\chi)\\
& \leq \Sigma_{n}(q,t,\chi)^{\frac{q_2-q_1}{q}}\Sigma_{n}(q_1,t,\chi).
\end{align*}
From this it follows:
\begin{align*}\limsup_{n \to +\infty} \frac{\log_2 (\Sigma_{n}(q_2,t,\chi))}{-n} \geq & \frac{q_2-q_1}{q}
\limsup_{n \to +\infty} \frac{\log_2 (\Sigma_{n}(q,t,\chi))}{-n} 
\\ &+ \limsup_{n \to +\infty} \frac{\log_2 (\Sigma_{n}(q_1,t,\chi))}{-n}
\end{align*}
which gives:
\[\limsup_{n \to +\infty} \frac{\log_2 (\Sigma_{n}(q,t,\chi))}{-n} \leq 
q \frac{q_2-q_1-\psi(q_2)+ \psi(q_1)}{q_2-q_1}.\] We now just have to
take the limit $q_1 \to q_\chi$ and $q_2 \to q_\chi$ to obtain the result.

\appendix

\section{Proof of Lemma \ref{lem:Auxilliaire}}
\label{App:lemAux}
  We first show {\bf (v)}. By setting $v=2^{n}u$, we have
\[2^n\int_{0}^{2^{-n}} e^{ \beta^{(q)}_{2^{-n},2^{-n}}(u)} du = 
\int_{0}^{1} e^{\beta^{(q)}_{2^{-n},2^{-n}}(2^{-n}v)} dv.\] It is straightforward to check that
\[\mu(\mathcal B_{2^{-n},2^{-n}}(u)) =
1-\frac{2^{m-n}}{1-2^{m-n}}.\] We denote this quantity by
$\nu_n$. Then if $\lambda^{(q)} = \bigl(\lambda^{(q)} (u) , \, u \geq 0 \bigr)$ is a L\'evy process such that
  $\mathbb E [e^{r\lambda^{(q)} (u)}] = e^{ \psi^{(q)}(r) u}$ fopr $r
\ge 0$, it will be easily 
  seen that we have the equality in distribution
  \begin{equation*}\bigl( \beta^{(q)}_{2^{-n},2^{-n}}(2^{-n}v) , \, 0 \leq  v  \leq 1 \bigr) \stackrel{d}{=}
  \bigl(\lambda^{(q)}   (\nu_n v+\nu_n) -
  \lambda^{(q)}( \nu_n v) ,  \, 0\leq v \leq 1 \bigr).\end{equation*}
 Observe that $\nu_n \to 1^-$ as $n \to +\infty$, and that we may
  apply  the  dominated convergence theorem from statement {\bf (i)} of the
  lemma, so that 
\[\mathbb E  \Biggl[\biggl| 1- 2^n\int_{0}^{2^{-n}}
    e^{\beta^{(q)}_{2^{-n},2^{-n}}(u)} du\biggr|^{1+\epsilon} \Biggr]
\to \mathbb E \Biggl[\biggl| 1- \int_{0}^{1} e^{\lambda^{(q)}_-(v+1)-\lambda^{(q)}_-(v)} dv \biggr|^{1+\epsilon}\Biggr]\]
as $n \to +\infty$, where $\lambda^{(q)}_-(v)$ is the left limit of $\lambda^{(q)}$ at time
$v$. Since from Assumption $\mathbf A_q$, $\psi^{(q)}(1+\epsilon)$ is
finite, the moment on the right hand side is as well finite.

The assertion {\bf (iv)} follows directly from {\bf (iii)} with
$r=q(1+\epsilon)$.

So as to show {\bf (iii)}, note that from the definition of the
$b_{n,k}$'s in equation \eqref{eqn:defb} and the $d_{n,k}$'s in equation \eqref{eqn:defd}, and
from the  definition of the constant
  $\gamma(r)$ in proposition \ref{thm:BM03moments}:
\begin{align} 
\gamma(r)  & = 2^{n(r-\psi(r))} T^{-\psi(r)} \mathbb E [
  b_{n,0}^r ]\nonumber \\
 &= 2^{n(r-\psi(r))} T^{-\psi(r)} \mathbb E [e^{r \tilde \omega_{2^{-n}}(0)}d_{n,0}] \nonumber \\
 &= \mathbb E \bigl[( 2^n d_{n,0} ) ^r\bigr] \label{eqn:gammar}
\end{align} 
from (\ref{eqn:muAnc}).
Using the definition of $c_{n,0}$ in equation \eqref{eqn:defc}, we see that we can bound the difference $|\gamma_n(r)-\gamma(r)|$ as:
\begin{equation*}
 |\gamma_n(r)-\gamma(r)| \leq  \mathbb E \biggl[c_{n,0}^r  \Bigl|1-   \frac
  {d_{n,0}^r}
  {c_{n,0}^r} \Bigr|\biggr].
\end{equation*}
The term in the absolute value on the right is dominated by \[\sup_{u
  \in [0, 2^{-n}]} |1-e^{r \theta_{2^{-n}}(u)}|\] which is independent of
   $c_{n,0}$.
  Therefore, statement {\bf (ii)} of the lemma together with definition
  \eqref{eqn:gamman} of $\gamma_n(r)$ show that:
\begin{equation*} |\gamma_n(r)-\gamma(r)| \lesssim \gamma_n(r)
  2^{-n\alpha} . \end{equation*}
 Moreover, one has from \eqref{eqn:gammar}:
\begin{eqnarray*}
\gamma(r) & \geq & \gamma_n(r) \mathbb E\biggl[ \inf_{u \in [0,2^{-n}]} e^{ r
\theta_{2^{-n}}(u)} \biggr],
\end{eqnarray*}
 and $\mathbb E\bigl[\inf_{u \in [0,2^{-n}]} e^{ r
\theta_{2^{-n}}(u)}\bigr]$ goes to
1 when $n$ goes to $\infty$, again from statement {\bf (ii)} of the lemma. Therefore
\[|\gamma_n(r)-\gamma(r)| \lesssim \gamma(r)  2^{-n\alpha},\]
hence  {\bf (iii)}.

 We now show {\bf (ii)}. The proof is the same for $\theta$ and $\theta^{(q)}$. For
 $u \in [0, 2^{-n}]$, we have:
  \[\mu(\mathcal
  T_{2^{-n}}(u)) = {2^{m- n}}/({1-2^{m- n}}).\] We denote by
  $\kappa_n$ this quantity. Since $m= \lfloor (1-\delta)n\rfloor$ for
  a fixed $\delta \in(0,1)$, $\kappa_n$ goes to 0 at an exponential
  rate as $n \to +\infty$.

  If $\lambda = \bigl( \lambda(u) , \, u \geq 0 \bigr)$ is a L\'evy process such that
  $\mathbb E [e^{r \lambda(u)}] = e^{ \psi(r) u}$, it will be easily 
  seen that we have the equality in distribution
  \begin{equation*}\bigl( \theta_{2^{-n}}(2^{-n}u) , \, 0 \leq  u  \leq 1 \bigr) \stackrel{d}{=}
  \bigl( \lambda  ( \kappa_n u+\kappa_n) -
  \lambda( \kappa_n u) ,  \, 0\leq u \leq 1 \bigr),\end{equation*}
  so that we may show the result on $\lambda$ rather than on $\theta_{2^{-n}}.$

  From the decomposition
  \[ \lambda  ( \kappa_n u+\kappa_n) -
  \lambda( \kappa_n u)=  \lambda  ( \kappa_n u+\kappa_n) -
  \lambda(\kappa_n)+ \lambda(\kappa_n)-
  \lambda( \kappa_n u),\] we have that
  \[ \mathbb E \bigl[ \sup_{v \in [0,\kappa_n]}
  |1-e^{r(\lambda(v+\kappa_n)-\lambda(v))}| \bigr] \le 
  \mathbb E \bigl[ \sup_{v_1, \, v_2 \in [0,\kappa_n]}
  |1-e^{r(\lambda_1(v_1)+\lambda_2(v_2))}|\bigr]
  \]
  where $\lambda_1$ and $\lambda_2$ are independent copies of
  $\lambda$. From the triangle inequality,
\begin{align*}
 \sup_{v_1, \, v_2 \in [0,\kappa_n]}
  |1-e^{r(\lambda_1(v_1)+\lambda_2(v_2))}| \le & \sup_{v_1, \, v_2 \in [0,\kappa_n]}
  |1-e^{r(\lambda_1(v_1)+\lambda_2(v_2)) -(v_1+v_2)\psi(r)}| \\& +
 \sup_{v_1, \, v_2 \in [0,\kappa_n]}
 e^{r(\lambda_1(v_1)+\lambda_2(v_2))} |1-e^{(v_1+v_2)\psi(r)}|
\end{align*}
 The term $ \mathbb E \sup_{v_1, \, v_2 \in [0,\kappa_n]}
 e^{r(\lambda_1(v_1)+\lambda_2(v_2))}$ is finite from {\bf (i)}, and since $\lambda$ is right-continuous, it goes
   to 1 as $n \to \infty$, so that
   \[\mathbb E \biggl[ \sup_{v_1, \, v_2 \in [0,\kappa_n]}
 e^{r(\lambda_1(v_1)+\lambda_2(v_2))}  |1-e^{(v_1+v_2)\psi(r)}|\biggr]
 \lesssim \kappa_n.\] 

 It is easily seen that for $x>0$, $\bigl(|1-xe^{r\lambda_1(v_1)-v_1\psi(r)} |, \,
 v_1 \ge 0\bigr)$ and $\bigl(|1-xe^{r\lambda_2(v_2)-v_2\psi(r)} |, \,
 v_2 \ge 0\bigr)$ are right-continuous submartingales. Thus, we may apply twice
 Doob's inequality, which gives
 \[\mathbb E \biggl[\sup_{v_1, \, v_2 \in [0,\kappa_n]}
  |1-e^{r(\lambda_1(v_1)+\lambda_2(v_2)) -(v_1+v_2)\psi(r)}| \biggr] 
   \le \mathbb
   E\bigl[|1-e^{r(\lambda_1(\kappa_n)+\lambda_2(\kappa_n))-2\kappa_n\psi(r)}|\bigr]. \] 
Since $\bigl(r(\lambda_1(v)+\lambda_2(v))-2v\psi(r), \, v \ge 0\bigr)$ is a L\'evy process, the L\'evy-Khintchine
   formula (see for instance the first chapter of Bertoin
   \cite{Ber96}) states
  that it can be written as 
  \[r(\lambda_1(v)+\lambda_2(v))-2v\psi(r) = av + \sigma  b(v) + x_1(v) + x_2(v) , \, v \geq
  0\] where $b$ is a standard Brownian motion, $x_1$ is a
  compound Poisson process with jumps of size greater than 1, $x_2$ is a
  pure-jump martingale with jumps of size less than 1, and $b$, $x_1$ and $x_2$ are independent. We first deal with the large jumps:
\begin{align*}
\mathbb E \bigl[|1- e^{a \kappa_n + \sigma b(\kappa_n)+x_1(\kappa_n)+x_2(\kappa_n)}|\bigr]
 &= \mathbb E \bigl[|1- e^{a \kappa_n + \sigma b(\kappa_n)+x_1(\kappa_n)+x_2(\kappa_n)}|1_{\{x_1(\kappa_n)\ne 0\}}\bigr]\\
&+\mathbb P\bigr[x_1(\kappa_n) = 0\bigl]\mathbb E \bigl[|1- e^{a \kappa_n + \sigma b(\kappa_n)+x_2(\kappa_n)}|\bigr]
\end{align*}
and                                                                                                                   \begin{align*}
& \mathbb E \bigl[|1- e^{a \kappa_n + \sigma b(\kappa_n)+x_1(\kappa_n)+x_2(\kappa_n)}|1_{\{x_1(\kappa_n)\ne 0\}}\bigr]\\
& \lesssim \mathbb P \bigr[x_1(\kappa_n) \ne 0\bigl] \mathbb E \bigl[|1- e^{a \kappa_n + \sigma b(\kappa_n)+x_1(\kappa_n)+x_2(\kappa_n)}|\bigr]\\
&\lesssim 1-e^{-\rho \kappa_n}\\
& \lesssim \kappa_n
\end{align*}
where $\rho$ is the intensity of the compound Poisson process $x_1$. Finally, the moment of order 2 of $e^{ \sigma b(\kappa_n)+x_2(\kappa_n)}$ is finite, and again from the L\'evy-Khintchine formula we have 
\[\mathbb E \bigl[ e^{a \kappa_n + \sigma b(\kappa_n)+x_2(\kappa_n)}\bigr] = e^{\phi(1)\kappa_n}\]
and 
\[\mathbb E \bigl[ e^{2(a \kappa_n + \sigma b(\kappa_n)+x_2(\kappa_n))}\bigr] = e^{\phi(2)\kappa_n}\]
for some real numbers $\phi(1)$ and $\phi(2)$.
Thus
\begin{align*}
 \mathbb E \bigl[|1- e^{a \kappa_n + \sigma b(\kappa_n)+x_2(\kappa_n)}|\bigr] 
& \le  \Bigl( \mathbb E \bigl[|1- e^{a \kappa_n + \sigma b(\kappa_n)+x_2(\kappa_n)}|^2\bigr] \Bigr)^{1/2} \\
& \le \bigl(1-2e^{\kappa_n \phi(1)} + e^{\kappa_n \phi(2)}\bigr)^{1/2}\\
& \lesssim \kappa_n^{1/2}.
\end{align*}

Let us show {\bf (i)}. The proof is the same for $P$ and $P^{(q)}$.
    We first suppose that $\mathcal C^m = \cap_{0 \le u \le
    t} \mathcal C+u$ is not empty. Then for $u \in [0,t]$, we decompose $\mathcal
    C+u$ into three disjoints sets:
\[\mathcal C+u = \mathcal C^m \cup \mathcal C^l(u) \cup
    \mathcal C^r(u),\] where $\mathcal C^{l}(u)$ is the part of
    $\mathcal C+u$ that is on the left of $\mathcal C^m$ and
    $\mathcal C^{r}(u)$ the part that is on the right. Then 
\[\bigr( P (\mathcal C^{l}(t-u)),\, 0 \le u \le t \bigl)
    \quad \text{and} \quad \bigr(P (\mathcal C^{r}(u)),\, 0 \le u
    \le t \bigl)  \]
are independent martingales, and they are also independent of
    $ P(\mathcal C^m)$. Thus, applying Doob's inequality,
\[\mathbb E\bigl[\sup_{0 \le u \le t} e^{r
    P(\mathcal C+u)} \bigr] \lesssim \mathbb E \bigl[e^{r P
    (\mathcal C^m)}\bigr]
   \mathbb E \bigl[e^{r P
    (\mathcal C^l(0))}\bigr]
\mathbb E \bigl[e^{r P
    (\mathcal C^r(t))}\bigr].
\]
Now recall from Assumption $\mathbf A_q$ that $\psi(r) < +\infty$ for
$r \le q(1+\epsilon)$, so that the last expression is indeed finite.
Finally, in the case where $t$ is large enough so that $\mathcal C^m$ is empty, we choose an
integer $j$ so that $\cap_{0 \le u \le t/j} \mathcal C+u \ne
\emptyset$, and we get
\[\mathbb E\bigl[\sup_{0 \le u \le t} e^{r
    P(\mathcal C+u)}\bigr]  \le j \mathbb E\bigl[\sup_{0 \le u \le t/j} e^{r
    P(\mathcal C+u)}\bigr].\]

\section{Proof of lemma \ref{lem:MajorationEMq}}\label{App:MajorationEMq}
We follow here  closely a proof given in \cite{BM03}. 
Let us decompose $M_{2^{-n}}^{(q,\epsilon)}(T)$ as
\[M_{2^{-n}}^{(q,\epsilon)}(T) = \sum_{k=0}^{2^{n_0-1}-1}e_{2k} +
\sum_{k=0}^{2^{n_0-1}-1}e_{2k+1}\]
where 
\[e_k = M_{2^{-n}}^{(q,\epsilon)} \bigl((k+1)T 2^{-n_0}\bigl) -
M_{2^{-n}}^{(q,\epsilon)} \bigr(kT 2^{-n_0}\bigr).\]
Thus,
\[\mathbb E \bigl[M_{2^{-n}}^{(q,\epsilon)}(T)^{1+\epsilon}\bigr] \leq
2^{1+\epsilon} \mathbb E \biggl[ \bigl( \sum_{k=0}^{2^{n_0-1}-1}e_{2k}
\bigr)^{1+\epsilon}\biggr].\]
We next apply the sub-additivity of the function $x \mapsto x^{(1+\epsilon)/2}$:
\begin{align*}
\mathbb E \bigl[M_{2^{-n}}^{(q,\epsilon)}(T)^{1+\epsilon}\bigr] &\leq
2^{1+\epsilon} \mathbb E \biggl[ \bigl( \sum_{k=0}^{2^{n_0-1}-1}e_{2k}^{(1+\epsilon)/2}
\bigr)^2\biggr]\\
& = 2^{\epsilon + n_0} \mathbb E \bigl[e_{0}^{1+\epsilon}\bigr] + 2^{1+\epsilon}
\sum_{k \neq k'} \mathbb E[e_{2k}^{(1+\epsilon)/2} e_{2k'}^{(1+\epsilon)/2}].
\end{align*}
If we now define $\omega_{l,L}^{(q)}(u) = \omega_{l}^{(q)}(u) -
\omega_{L}^{(q)}(u)$ for $0<l<L$, then we can write:
\begin{equation} \label{eqn:beaute}
e_{2k} = 2^{-n \chi \epsilon/(1+\epsilon)} \int_{2kT2^{-n_0}}^{2(k+1)T2^{-n_0}}
e^{\omega_{2^{-n},T2^{-n_0}}^{(q)}(u) + \omega_{T2^{-n_0}}^{(q)}(u)} du 
\end{equation}
so that
\[e_{2k} \leq 2^{-n \chi \epsilon/(1+\epsilon)} \biggl( \sup_{v \in
  [0,T]}e^{\omega_{T2^{-n_0}}^{(q)}(v)}  \biggr) \int_{2kT2^{-n_0}}^{2(k+1)T2^{-n_0}}
e^{\omega_{2^{-n},T2^{-n_0}}^{(q)}(u) }du.\]
In this last inequality, the sup term is independent of the
integral. Moreover, two integral terms for different values of $k$ are also
independent. 
Let us now
remark that from statement {\bf (i)} of Lemma \ref{lem:Auxilliaire}, there exists a constant $C>0$ (that depends on
$n_0$) such that
\begin{equation} \label{eqn:sup}
\mathbb E \biggl[\sup_{v \in [0,T]} e^{(1+\epsilon)\omega_{T2^{-n_0}}^{(q)}
  (u)}\biggr] \leq C.
\end{equation}
Using Jensen's inequality, we can therefore write for $k \ne k'$:
\begin{align*} \mathbb E[e_{2k}^{(1+\epsilon)/2}
  e_{2k'}^{(1+\epsilon)/2}] & \leq 2^{-n \chi \epsilon}C \biggl( \mathbb E \Bigl[ \Big(\int_{0}^{T2^{-n_0}}
e^{\omega_{2^{-n},T2^{-n_0}}^{(q)}(u) }du \Big)^{(1+\epsilon)/2}\Bigr]\biggr)^{2}\\
& \leq 2^{-n \chi \epsilon}C \biggl( \mathbb E \Bigl[ \int_{0}^{T2^{-n_0}}
e^{\omega_{2^{-n},T2^{-n_0}}^{(q)}(u) }du\Bigr]\biggr)^{1+\epsilon} \\
& \le 2^{-n \chi \epsilon-n_0(1+\epsilon)}CT^{1+\epsilon}\\
& \lesssim 2^{-n \chi \epsilon}.
\end{align*}
This proves the lemma.

\end{document}